\documentclass[a4paper,10pt]{amsart}

\usepackage[utf8]{inputenc}
\usepackage[hidelinks]{hyperref}

\usepackage{microtype}
\usepackage{amssymb,amsmath,amsopn,amsxtra,amsthm,amsfonts}
\usepackage{xr}
\usepackage[mathcal]{euscript}
\let\mathscr\mathcal
\usepackage[bb=px]{mathalfa}

\usepackage[english, status=final, nomargin, inline]{fixme}
\fxusetheme{color}

\newtheorem{thm}[subsubsection]{Theorem}
\newtheorem{thmX}{Theorem}

\newtheorem{cor}[subsubsection]{Corollary}
\newtheorem{lem}[subsubsection]{Lemma}
\newtheorem{prop}[subsubsection]{Proposition}

\theoremstyle{definition}
\newtheorem{constr}[subsubsection]{Construction}

\newtheorem{rem}[subsubsection]{Remark}
\newtheorem{defn}[subsubsection]{Definition}
\newtheorem{warn}[subsubsection]{Warning}
\newtheorem{exam}[subsubsection]{Example}
\newtheorem{notat}[subsubsection]{Notation}

\newcommand{\constrref}[1]{Construction~\ref{#1}}
\newcommand{\thmref}[1]{Theorem~\ref{#1}}

\newcommand{\secref}[1]{Sect.~\ref{#1}}
\newcommand{\ssecref}[1]{Subsect.~\ref{#1}}

\newcommand{\lemref}[1]{Lemma~\ref{#1}}
\newcommand{\propref}[1]{Proposition~\ref{#1}}
\newcommand{\corref}[1]{Corollary~\ref{#1}}

\newcommand{\remref}[1]{Remark~\ref{#1}}
\newcommand{\defref}[1]{Definition~\ref{#1}}
\newcommand{\examref}[1]{Example~\ref{#1}}

\newcommand{\warnref}[1]{Warning~\ref{#1}}
\renewcommand{\eqref}[1]{(\ref{#1})}

\numberwithin{equation}{section}

\usepackage{tikz}
\usetikzlibrary{matrix}
\usepackage{tikz-cd}

\usepackage{xspace}

\usepackage{xifthen}

\usepackage{xparse}

\usepackage{mathtools}

\newcommand{\nc}{\newcommand}
\nc{\renc}{\renewcommand}
\nc{\ssec}{\subsection}
\nc{\sssec}{\subsubsection}
\nc{\on}{\operatorname}
\nc{\term}[1]{#1\xspace}

\nc{\sA}{\ensuremath{\mathscr{A}}\xspace}
\nc{\sB}{\ensuremath{\mathscr{B}}\xspace}
\nc{\sC}{\ensuremath{\mathscr{C}}\xspace}
\nc{\sD}{\ensuremath{\mathscr{D}}\xspace}
\nc{\sE}{\ensuremath{\mathscr{E}}\xspace}
\nc{\sF}{\ensuremath{\mathscr{F}}\xspace}
\nc{\sG}{\ensuremath{\mathscr{G}}\xspace}
\nc{\sH}{\ensuremath{\mathscr{H}}\xspace}
\nc{\sI}{\ensuremath{\mathscr{I}}\xspace}
\nc{\sJ}{\ensuremath{\mathscr{J}}\xspace}
\nc{\sK}{\ensuremath{\mathscr{K}}\xspace}
\nc{\sL}{\ensuremath{\mathscr{L}}\xspace}
\nc{\sM}{\ensuremath{\mathscr{M}}\xspace}
\nc{\sN}{\ensuremath{\mathscr{N}}\xspace}
\nc{\sO}{\ensuremath{\mathscr{O}}\xspace}
\nc{\sP}{\ensuremath{\mathscr{P}}\xspace}
\nc{\sQ}{\ensuremath{\mathscr{Q}}\xspace}
\nc{\sR}{\ensuremath{\mathscr{R}}\xspace}
\nc{\sS}{\ensuremath{\mathscr{S}}\xspace}
\nc{\sT}{\ensuremath{\mathscr{T}}\xspace}
\nc{\sU}{\ensuremath{\mathscr{U}}\xspace}
\nc{\sV}{\ensuremath{\mathscr{V}}\xspace}
\nc{\sW}{\ensuremath{\mathscr{W}}\xspace}
\nc{\sX}{\ensuremath{\mathscr{X}}\xspace}
\nc{\sY}{\ensuremath{\mathscr{Y}}\xspace}
\nc{\sZ}{\ensuremath{\mathscr{Z}}\xspace}

\nc{\bA}{\ensuremath{\mathbf{A}}\xspace}
\nc{\bB}{\ensuremath{\mathbf{B}}\xspace}
\nc{\bC}{\ensuremath{\mathbf{C}}\xspace}
\nc{\bD}{\ensuremath{\mathbf{D}}\xspace}
\nc{\bE}{\ensuremath{\mathbf{E}}\xspace}
\nc{\bF}{\ensuremath{\mathbf{F}}\xspace}
\nc{\bG}{\ensuremath{\mathbf{G}}\xspace}
\nc{\bH}{\ensuremath{\mathbf{H}}\xspace}
\nc{\bI}{\ensuremath{\mathbf{I}}\xspace}
\nc{\bJ}{\ensuremath{\mathbf{J}}\xspace}
\nc{\bK}{\ensuremath{\mathbf{K}}\xspace}
\nc{\bL}{\ensuremath{\mathbf{L}}\xspace}
\nc{\bM}{\ensuremath{\mathbf{M}}\xspace}
\nc{\bN}{\ensuremath{\mathbf{N}}\xspace}
\nc{\bO}{\ensuremath{\mathbf{O}}\xspace}
\nc{\bP}{\ensuremath{\mathbf{P}}\xspace}
\nc{\bQ}{\ensuremath{\mathbf{Q}}\xspace}
\nc{\bR}{\ensuremath{\mathbf{R}}\xspace}
\nc{\bS}{\ensuremath{\mathbf{S}}\xspace}
\nc{\bT}{\ensuremath{\mathbf{T}}\xspace}
\nc{\bU}{\ensuremath{\mathbf{U}}\xspace}
\nc{\bV}{\ensuremath{\mathbf{V}}\xspace}
\nc{\bW}{\ensuremath{\mathbf{W}}\xspace}
\nc{\bX}{\ensuremath{\mathbf{X}}\xspace}
\nc{\bY}{\ensuremath{\mathbf{Y}}\xspace}
\nc{\bZ}{\ensuremath{\mathbf{Z}}\xspace}

\nc{\dA}{\ensuremath{\mathds{A}}\xspace}
\nc{\dB}{\ensuremath{\mathds{B}}\xspace}
\nc{\dC}{\ensuremath{\mathds{C}}\xspace}
\nc{\dD}{\ensuremath{\mathds{D}}\xspace}
\nc{\dE}{\ensuremath{\mathds{E}}\xspace}
\nc{\dF}{\ensuremath{\mathds{F}}\xspace}
\nc{\dG}{\ensuremath{\mathds{G}}\xspace}
\nc{\dH}{\ensuremath{\mathds{H}}\xspace}
\nc{\dI}{\ensuremath{\mathds{I}}\xspace}
\nc{\dJ}{\ensuremath{\mathds{J}}\xspace}
\nc{\dK}{\ensuremath{\mathds{K}}\xspace}
\nc{\dL}{\ensuremath{\mathds{L}}\xspace}
\nc{\dM}{\ensuremath{\mathds{M}}\xspace}
\nc{\dN}{\ensuremath{\mathds{N}}\xspace}
\nc{\dO}{\ensuremath{\mathds{O}}\xspace}
\nc{\dP}{\ensuremath{\mathds{P}}\xspace}
\nc{\dQ}{\ensuremath{\mathds{Q}}\xspace}
\nc{\dR}{\ensuremath{\mathds{R}}\xspace}
\nc{\dS}{\ensuremath{\mathds{S}}\xspace}
\nc{\dT}{\ensuremath{\mathds{T}}\xspace}
\nc{\dU}{\ensuremath{\mathds{U}}\xspace}
\nc{\dV}{\ensuremath{\mathds{V}}\xspace}
\nc{\dW}{\ensuremath{\mathds{W}}\xspace}
\nc{\dX}{\ensuremath{\mathds{X}}\xspace}
\nc{\dY}{\ensuremath{\mathds{Y}}\xspace}
\nc{\dZ}{\ensuremath{\mathds{Z}}\xspace}

\nc{\bbA}{\ensuremath{\mathbb{A}}\xspace}
\nc{\bbB}{\ensuremath{\mathbb{B}}\xspace}
\nc{\bbC}{\ensuremath{\mathbb{C}}\xspace}
\nc{\bbD}{\ensuremath{\mathbb{D}}\xspace}
\nc{\bbE}{\ensuremath{\mathbb{E}}\xspace}
\nc{\bbF}{\ensuremath{\mathbb{F}}\xspace}
\nc{\bbG}{\ensuremath{\mathbb{G}}\xspace}
\nc{\bbH}{\ensuremath{\mathbb{H}}\xspace}
\nc{\bbI}{\ensuremath{\mathbb{I}}\xspace}
\nc{\bbJ}{\ensuremath{\mathbb{J}}\xspace}
\nc{\bbK}{\ensuremath{\mathbb{K}}\xspace}
\nc{\bbL}{\ensuremath{\mathbb{L}}\xspace}
\nc{\bbM}{\ensuremath{\mathbb{M}}\xspace}
\nc{\bbN}{\ensuremath{\mathbb{N}}\xspace}
\nc{\bbO}{\ensuremath{\mathbb{O}}\xspace}
\nc{\bbP}{\ensuremath{\mathbb{P}}\xspace}
\nc{\bbQ}{\ensuremath{\mathbb{Q}}\xspace}
\nc{\bbR}{\ensuremath{\mathbb{R}}\xspace}
\nc{\bbS}{\ensuremath{\mathbb{S}}\xspace}
\nc{\bbT}{\ensuremath{\mathbb{T}}\xspace}
\nc{\bbU}{\ensuremath{\mathbb{U}}\xspace}
\nc{\bbV}{\ensuremath{\mathbb{V}}\xspace}
\nc{\bbW}{\ensuremath{\mathbb{W}}\xspace}
\nc{\bbX}{\ensuremath{\mathbb{X}}\xspace}
\nc{\bbY}{\ensuremath{\mathbb{Y}}\xspace}
\nc{\bbZ}{\ensuremath{\mathbb{Z}}\xspace}

\nc{\cA}{\ensuremath{\mathcal{A}}\xspace}
\nc{\cB}{\ensuremath{\mathcal{B}}\xspace}
\nc{\cC}{\ensuremath{\mathcal{C}}\xspace}
\nc{\cD}{\ensuremath{\mathcal{D}}\xspace}
\nc{\cE}{\ensuremath{\mathcal{E}}\xspace}
\nc{\cF}{\ensuremath{\mathcal{F}}\xspace}
\nc{\cG}{\ensuremath{\mathcal{G}}\xspace}
\nc{\cH}{\ensuremath{\mathcal{H}}\xspace}
\nc{\cI}{\ensuremath{\mathcal{I}}\xspace}
\nc{\cJ}{\ensuremath{\mathcal{J}}\xspace}
\nc{\cK}{\ensuremath{\mathcal{K}}\xspace}
\nc{\cL}{\ensuremath{\mathcal{L}}\xspace}
\nc{\cM}{\ensuremath{\mathcal{M}}\xspace}
\nc{\cN}{\ensuremath{\mathcal{N}}\xspace}
\nc{\cO}{\ensuremath{\mathcal{O}}\xspace}
\nc{\cP}{\ensuremath{\mathcal{P}}\xspace}
\nc{\cQ}{\ensuremath{\mathcal{Q}}\xspace}
\nc{\cR}{\ensuremath{\mathcal{R}}\xspace}
\nc{\cS}{\ensuremath{\mathcal{S}}\xspace}
\nc{\cT}{\ensuremath{\mathcal{T}}\xspace}
\nc{\cU}{\ensuremath{\mathcal{U}}\xspace}
\nc{\cV}{\ensuremath{\mathcal{V}}\xspace}
\nc{\cW}{\ensuremath{\mathcal{W}}\xspace}
\nc{\cX}{\ensuremath{\mathcal{X}}\xspace}
\nc{\cY}{\ensuremath{\mathcal{Y}}\xspace}
\nc{\cZ}{\ensuremath{\mathcal{Z}}\xspace}

\nc{\mrm}[1]{\ensuremath{\mathrm{#1}}\xspace}
\nc{\mbf}[1]{\ensuremath{\mathbf{#1}}\xspace}
\nc{\mcal}[1]{\ensuremath{\mathcal{#1}}\xspace}
\nc{\mit}[1]{\ensuremath{\mathit{#1}}\xspace}
\nc{\msc}[1]{\ensuremath{\mathscr{#1}}\xspace}

\renc{\bar}[1]{\overline{#1}}

\let\sectsign\S
\let\S\relax

\nc{\sub}{\subset}
\nc{\too}{\longrightarrow}
\nc{\hook}{\hookrightarrow}
\nc*{\hooklongrightarrow}{\ensuremath{\lhook\joinrel\relbar\joinrel\rightarrow}}
\nc{\hooklong}{\hooklongrightarrow}
\nc{\twoheadlongrightarrow}{\relbar\joinrel\twoheadrightarrow}
\nc{\shiso}{\approx}
\nc{\isoto}{\xrightarrow{\sim}}
\nc{\isofrom}{\xleftarrow{\sim}}
\renc{\ge}{\geqslant}
\renc{\le}{\leqslant}

\nc{\id}{\mathrm{id}}

\DeclareMathOperator{\Hom}{\mathrm{Hom}}
\nc{\uHom}{\underline{\smash{\Hom}}}
\DeclareMathOperator{\Map}{\mathrm{Maps}}

\DeclareMathOperator{\End}{\mathrm{End}}
\DeclareMathOperator{\Sym}{\mathrm{Sym}}
\nc{\uEnd}{\underline{\smash{\End}}}

\nc{\colim}{\varinjlim}
\renc{\lim}{\varprojlim}
\nc{\Cofib}{\on{Cofib}}
\nc{\Fib}{\on{Fib}}
\nc{\initial}{\varnothing}
\nc{\op}{\mathrm{op}}

\DeclareMathOperator*{\fibprod}{\times}
\DeclareMathOperator*{\fibcoprod}{\operatorname{\sqcup}}
\renc{\coprod}{\sqcup}

\nc{\Spc}{{\mrm{Spc}}}
\nc{\pt}{{\mrm{pt}}}
\nc{\PSh}{\mrm{PSh}}
\nc{\h}{\on{h}}
\nc{\Sh}{\mrm{Sh}}
\renc{\L}{\mrm{L}}

\nc{\CAlg}{\mrm{CAlg}}
\nc{\cn}{{\mrm{cn}}}

\nc{\Spec}{\on{Spec}}
\nc{\A}{\bA}
\nc{\Sm}{\mrm{Sm}}
\nc{\Et}{\mrm{Et}}
\nc{\cl}{{\mrm{cl}}}
\nc{\red}{{\mrm{red}}}
\nc{\aff}{{\mrm{aff}}}
\nc{\Zar}{\mathrm{Zar}}
\nc{\Nis}{\mathrm{Nis}}
\nc{\et}{\mathrm{\acute et}}
\nc{\sm}{\mathrm{sm}}

\nc{\MotSpc}{{\mbf{H}}}
\nc{\LNis}{\on{\L_\Nis}}
\nc{\htp}{{\A^1}}
\nc{\Lhtp}{\on{\L_\htp}}
\nc{\mot}{{\mrm{mot}}}
\nc{\Lmot}{\on{\L_\mot}}
\nc{\hspc}[2][]{\h_{#1}(#2)}
\nc{\hmot}[2][]{\on{M}_{#1}(#2)}

\nc{\Einfty}{{\sE_\infty}}
\nc{\bDelta}{\mbf{\Delta}}
\nc{\Tot}{\on{Tot}}
\nc{\Cech}{\textnormal{\v{C}}}
\nc{\K}{\on{K}}

\nc{\cEring}{\term{connective $\Einfty$-ring}}
\nc{\cErings}{\term{connective $\Einfty$-rings}}

\nc{\inftyCat}{\term{$\infty$-category}}
\nc{\inftyCats}{\term{$\infty$-categories}}

\nc{\inftyGrpd}{\term{$\infty$-groupoid}}
\nc{\inftyGrpds}{\term{$\infty$-groupoids}}

\title{The~Morel--Voevodsky~localization~theorem in~spectral~algebraic~geometry}
\author{Adeel A. Khan}
\date{\today}
\address{Fakultät für Mathematik\\
Universität Regensburg\\
93040 Regensburg\\
Germany}
\email{\href{mailto:adeel.khan@mathematik.uni-regensburg.de}{adeel.khan@mathematik.uni-regensburg.de}}
\urladdr{\url{https://www.preschema.com}}

\begin{document}

\begin{abstract}
We prove an analogue of the Morel--Voevodsky localization theorem over spectral algebraic spaces.
As a corollary we deduce a ``derived nilpotent invariance'' result which, informally speaking, says that $\mathbf{A}^1$-homotopy invariance kills all higher homotopy groups of a connective commutative ring spectrum.
\end{abstract}

\maketitle

\parskip 0.2cm


\section{Introduction}
\label{sec:intro}

Let $R$ be a connective $\Einfty$-ring spectrum, and denote by $\CAlg^\et_R$ the \inftyCat of étale $\Einfty$-algebras over $R$.
The starting point for this paper is the following fundamental result of J. Lurie, which says that the small étale topos of $R$ is equivalent to the small étale topos of $\pi_0(R)$ (see \cite[Thm.~7.5.0.6]{HA-20170918} and \cite[Rem.~B.6.2.7]{SAG-20180204}):

\begin{thm}[Lurie]\label{thm:intro/etale nil}
For any \cEring $R$, let $\pi_0(R)$ denote its $0$-truncation (viewed as a discrete $\Einfty$-ring).
Then restriction along the canonical functor $\CAlg^\et_R \to \CAlg^\et_{\pi_0(R)}$ induces an equivalence from the \inftyCat of étale sheaves of spaces $\CAlg^\et_{\pi_0(R)} \to \Spc$ to the \inftyCat of étale sheaves of spaces $\CAlg^\et_{R} \to \Spc$.
\end{thm}

\thmref{thm:intro/etale nil} can be viewed as a special case of the following result (see \cite[Prop.~3.1.4.1]{SAG-20180204}):

\begin{thm}[Lurie]\label{thm:intro/etale kashiwara}
Let $R \to R'$ be a homomorphism of \cErings that is surjective on $\pi_0$.
Then restriction along the canonical functor $\CAlg^\et_R \to \CAlg^\et_{R'}$ defines a fully faithful embedding of the \inftyCat of étale sheaves $\CAlg^\et_{R'} \to \Spc$ into the \inftyCat of étale sheaves $\CAlg^\et_{R} \to \Spc$.
Moreover, a sheaf $\sF$ belongs to the essential image if and only if its restriction to the complement of the closed subset $\Spec(R') \subseteq \Spec(R)$ is (weakly) contractible.
\end{thm}

Theorems \ref{thm:intro/etale nil} and \ref{thm:intro/etale kashiwara} are particular to small sites: for example, they do not hold for sheaves on the big site $\CAlg^\sm_R$ of \emph{smooth} $R$-algebras.
Our main objective in this paper is to show that, if we restrict to sheaves that are \emph{$\A^1$-homotopy invariant}, then these results do have analogues on the big sites (and we can even replace the étale topology by the coarser Nisnevich topology).
To be precise, we have (see \corref{cor:nilpotent invariance affine} and \thmref{thm:kashiwara}):

\begin{thmX}\label{thm:intro/A^1 nilpotent}
For any \cEring $R$, restriction along the canonical functor $\CAlg^\sm_R \to \CAlg^\sm_{\pi_0(R)}$ induces an equivalence from the \inftyCat of $\A^1$-homotopy invariant Nisnevich sheaves $\CAlg^\sm_{\pi_0(R)} \to \Spc$ to the \inftyCat of $\A^1$-homotopy invariant Nisnevich sheaves $\CAlg^\sm_{R} \to \Spc$.
\end{thmX}

\begin{thmX}\label{thm:intro/kashiwara}
Let $i : Z \to S$ be a closed immersion of quasi-compact quasi-separated spectral algebraic spaces, with quasi-compact open complement $j : U \hook S$.
Denote by $\Sm_{/S}$, resp. $\Sm_{/Z}$, the \inftyCat of smooth spectral algebraic spaces over $S$, resp. $Z$.
Then the direct image functor $i_*$ defines a fully faithful embedding of the \inftyCat of $\A^1$-invariant Nisnevich sheaves on $\Sm_{/Z}$ into the \inftyCat of $\A^1$-invariant Nisnevich sheaves on $\Sm_{/S}$.
Moreover, an object $\sF$ belongs to the essential image if and only if its inverse image $j^*(\sF)$ is (weakly) contractible.
\end{thmX}

\thmref{thm:intro/kashiwara} can also be viewed as an analogue of Kashiwara's lemma in D-modules (as generalized in \cite{GaitsgoryRozenblyumCrystals} to the setting of spectral algebraic geometry over fields of characteristic zero).
It is essentially a reformulation of our main result, an analogue of the localization theorem of Morel--Voevodsky \cite[Thm.~3.2.21]{MorelVoevodsky} in the setting of spectral algebraic geometry.
By analogy with \emph{op. cit.}, we define a \emph{motivic space} over a spectral algebraic space $S$ as an $\A^1$-invariant Nisnevich sheaf of spaces on $\Sm_{/S}$.
Then we have (see \thmref{thm:localization}):

\begin{thmX}[Localization]\label{thm:intro/localization}
Let $i : Z \to S$ be a closed immersion of quasi-compact quasi-separated spectral algebraic spaces, with quasi-compact open complement $j : U \hook S$.
Let $j_\sharp$ denote the ``extension by zero'' functor, left adjoint to $j^*$.
Then for any motivic space over $S$, there is a cocartesian square
  \begin{equation*}
    \begin{tikzcd}
      j_\sharp j^* (\sF) \ar{r}\ar{d}
        & \sF \ar{d} \\
      j_\sharp (\pt_U) \ar{r}
        & i_* i^* (\sF)
    \end{tikzcd}
  \end{equation*}
of motivic spaces over $S$.
\end{thmX}

\ssec{Outline}

In order to make sense of \thmref{thm:intro/A^1 nilpotent}, we need to define the notions of \emph{smoothness} and of \emph{$\A^1$-homotopy invariance} in the world of $\Einfty$-ring spectra.
There are two natural ways to define smoothness for a homomorphism of \cErings $A \to B$:
\begin{itemize}
  \item One can require that $B$ is flat as an $A$-module, and that the induced homomorphism of ordinary commutative rings $\pi_0(A) \to \pi_0(B)$ is smooth.

  \item One can require that $B$ is locally of finite presentation as an $A$-algebra, and that the relative cotangent complex $L_{B/A}$ is a finitely generated projective $B$-module.
\end{itemize}
There are also two candidate ``affine lines'' over a \cEring $R$:
\begin{itemize}
  \item The ``flat affine line'' $\A^1_{\flat,\Spec(R)}$ (\remref{rem:flat A^n}), whose $\Einfty$-ring of functions is the polynomial $\Einfty$-algebra $R[T] = R \otimes \Sigma^\infty_+(\bN)$.
  This affine line is smooth in the first sense, and is compatible with the affine line in classical algebraic geometry.
  That is, when $R$ is discrete, $\A^1_{\flat,\Spec(R)}$ is the classical affine line over $R$.

  \item The ``spectral affine line'' $\A^1_{\Spec(R)}$ (\examref{exam:vector bundles}), whose $\Einfty$-ring of functions is the free $\Einfty$-algebra $R\{T\}$ on one generator $T$ (in degree zero).
  This spectral algebraic space is smooth in the second sense, and represents the functor sending an $R$-algebra $A$ to its underlying space $\Omega^\infty(A)$.
\end{itemize}
In this paper we work with the second definition of smoothness, and with the ``intrinsic'' spectral affine line $\A^1_{\Spec(R)}$.
We review the appropriate definitions in detail in \secref{sec:motspc}.
In the setting of derived algebraic geometry (formed out of simplicial commutative rings), the two affine lines collapse into one, so that the resulting $\A^1$-homotopy theory is a much simpler version of the theory developed here (see the author's Ph.D. thesis \cite{KhanThesis}).
In case $R$ is of characteristic zero (an $\Einfty$-$\bQ$-algebra), the theory of spectral algebraic geometry over $R$ is equivalent to derived algebraic geometry over $R$, and the $\A^1$-homotopy theory constructed here recovers the construction of \emph{op. cit}.
Over a general $R$, we have two different affine lines and two \emph{a priori} different versions of $\A^1$-homotopy theory (see \warnref{warn:char p}).

In \secref{sec:results} we turn to our main results, which are all centred around the functor $i_*$ of direct image along a closed immersion $i$.
We begin by proving that $i_*$ commutes with almost all colimits (\thmref{thm:i_* commutes with contractible colimits}).
We then state the localization theorem (\thmref{thm:intro/localization} above) as \thmref{thm:localization}, postponing its proof to \secref{sec:loc}.
We first explain how it implies \thmref{thm:intro/kashiwara} (\thmref{thm:kashiwara}) and \thmref{thm:intro/A^1 nilpotent} (\corref{cor:nilpotent invariance affine}).
As another application, we then proceed to develop part of the formalism of Grothendieck's six operations: the proper base change and projection formulas in the case of closed immersions (Propositions~\ref{prop:closed base change} and \ref{prop:closed projection formula}) and a smooth-closed base change formula (\propref{prop:(smooth,closed)-base change}).

Finally, \secref{sec:loc} is dedicated to the proof of \thmref{thm:intro/localization}.
Aside from generalizing the theorem of Morel--Voevodsky \cite{MorelVoevodsky} to the spectral setting, our statement also differs in a couple other (mutually orthogonal) ways:
\begin{itemize}
  \item We do not impose noetherian hypotheses.
  For this reason, we give a proof of \thmref{thm:intro/localization} that avoids the use of ``points'' and therefore applies to sheaves satisfying \emph{\v{C}ech descent}, as opposed to the (\emph{a priori}) stronger condition of hyperdescent (see \remref{rem:hyper}).
  An alternative approach to removing noetherian hypotheses is to use continuity arguments to reduce to the noetherian case, as described in the classical setting in \cite[App.~C]{HoyoisLefschetz}.

  \item We generalize the result from (spectral) schemes to (spectral) algebraic spaces.
  The key point is that every quasi-compact quasi-separated algebraic space is Nisnevich-locally affine (see \cite[Chap.~II, Thm.~6.4]{KnutsonAlgebraicSpaces}).
  To be precise, one needs a little more than this: see the proof of \propref{prop:Spc_Nis generated by affines}.
  Repeating the proof of \thmref{thm:intro/localization} in the setting of classical algebraic geometry, one can similarly generalize the statement of \cite[Thm.~3.2.21]{MorelVoevodsky} to algebraic spaces.
\end{itemize}
Our proof follows the same general strategy as the original proof of Morel--Voevodsky, but differs in some details.
Let $i : Z \hook S$ be a closed immersion as in the statement.
The first step is to use \propref{prop:Spc_Nis generated by affines} and \thmref{thm:i_* commutes with contractible colimits} to reduce to the case of (the motivic localization of) a sheaf represented by a smooth spectral algebraic space $X$ over the base $S$.
Then given a partially defined section $t : Z \hook X$ over $S$, we have to show that a certain presheaf $\h_S(X,t)$ is motivically contractible.
We achieve this in a few steps:
\begin{itemize}
  \item Nisnevich-locally on $X$, we can lift the partially defined section $t : Z \hook X$ to a section $s : S \hook X$ (\lemref{lem:lifting sections}).
  \item Nisnevich-locally on $X$, the section $s$ can be approximated by the zero section of a trivial vector bundle on $S$, up to some étale morphism that induces an isomorphism over $S$ (\lemref{lem:linear approximation}).
  Moreover, the construction $\h_S(X,t)$ is invariant under such approximations (\lemref{lem:etale invariance of h(X,t)}).
  \item If $X$ is a vector bundle over $S$ (and $t$ is the restriction of the zero section), then $\h_S(X,t)$ is $\A^1$-contractible (\lemref{lem:h(E,s) of vector bundle}).
\end{itemize}

\ssec{Notation and conventions}

We will use the language of \inftyCats freely throughout the text.
Our main references are \cite{HTT,HA-20170918}.
The \inftyCat of spaces will be denoted by $\Spc$, and a morphism in an \inftyCat will be called an \emph{isomorphism} if it is invertible (= an \emph{equivalence} in the language of \cite{HTT}).

The term \emph{spectral algebraic space} will mean a quasi-compact quasi-separated spectral algebraic space as defined in \cite{SAG-20180204}.
An affine spectral scheme is a spectral algebraic space of the form $\Spec(R)$, where $R$ is a \cEring (see e.g.~\cite{HA-20170918}).
Any spectral algebraic space $S$ admits a finite Nisnevich covering by affine spectral schemes \cite[Ex. 3.7.1.5]{SAG-20180204}; it is a (quasi-compact quasi-separated) spectral \emph{scheme} in the sense of \cite{SAG-20180204} if and only if it moreover admits a \emph{Zariski} covering by finitely many affines.
It is a (quasi-compact quasi-separated) \emph{classical} algebraic space if and only if it admits a Nisnevich covering by finitely many \emph{classical} affines (of the form $\Spec(R)$ with $R$ discrete).
Given a spectral algebraic space $S$, we write $S_\cl$ for its \emph{underlying classical algebraic space}, so that $\Spec(R)_\cl = \Spec(\pi_0(R))$ for any \cEring $R$.

\ssec{Acknowledgments}

The author would like to thank Benjamin Antieau, Denis-Charles Cisinski, David Gepner, Marc Hoyois, and Marc Levine for many helpful discussions, encouragement, and feedback on previous versions of this paper.


\section{Motivic spaces}
\label{sec:motspc}

\ssec{Sm-fibred spaces}

\begin{defn}
Let $f : X \to S$ be a morphism of spectral algebraic spaces.
We say that $f$ is \emph{smooth} if it is of finite presentation and the relative cotangent complex $\sL_{X/S}$ is locally free of finite rank.
If moreover the cotangent complex vanishes, then we say that $f$ is \emph{étale}.
\end{defn}

\begin{exam}\label{exam:vector bundles}
Let $\bS$ denote the sphere spectrum, and $\bS\{T_1,\ldots,T_n\}$ the free $\Einfty$-algebra on $n$ generators (in degree zero).
Given a spectral algebraic space $S$, consider for any $n\ge 0$ the \emph{$n$-dimensional spectral affine space}
  \begin{equation*}
    \A^n_S = S \times \Spec(\bS\{T_1,\ldots,T_n\}).
  \end{equation*}
Then the projection $\A^n_S \to S$ has cotangent complex free of rank $n$, and is smooth.
More generally, if $\sE$ is a locally free sheaf of finite rank on $S$, then the associated vector bundle $\pi : \Spec_S(\Sym_{\sO_S}(\sE))\to S$ has relative cotangent complex $\pi^*(\sE)$, and is again smooth.
\end{exam}

\begin{rem}\label{rem:smooth structure}
Nisnevich-locally on $X$, any smooth morphism of spectral algebraic spaces $f : X \to S$ can be factored through an étale morphism $X \to \A^n_S$ and the projection $\A^n_S \to S$ (see \cite[Prop.~11.2.2.1]{SAG-20180204}).
\end{rem}

\begin{warn}
Unlike in classical algebraic geometry, smooth morphisms in spectral algebraic geometry are generally not flat: étale morphisms are flat, but $\A^n_S \to S$ is flat iff $n=0$ or $S$ is of characteristic zero.
In particular, if $S$ is classical, the spectral algebraic space $\A^n_S$ is classical iff $n=0$ or $S$ is of characteristic zero (cf. \warnref{warn:char p}).
\end{warn}

\begin{rem}\label{rem:flat A^n}
There is a variant of the construction $\A^n_S$ that is flat over $S$ (but usually not smooth).
Namely, let $\bS[T_1,\ldots,T_n]$ denote the polynomial $\Einfty$-algebra on $n$ generators over $\bS$ (in degree zero); this is by definition the suspension spectrum $\Sigma^\infty_+(\bN^n)$, where the (additive) commutative monoid $\bN^n$ is viewed as a discrete $\Einfty$-space.
If we set
  \begin{equation*}
    \A^n_{\flat,S} = S \times \Spec(\bS[T_1,\ldots,T_n]),
  \end{equation*}
then the projection $\A^n_{\flat,S} \to S$ is flat.
\end{rem}

\begin{defn}
Let $S$ be a spectral algebraic space.
A \emph{$\Sm$-fibred space} over $S$, or simply a \emph{fibred space} over $S$, is a presheaf of spaces on the \inftyCat $\Sm_{/S}$ of smooth spectral algebraic spaces over $S$.
We write $\Spc(S)$ for the \inftyCat of $\Sm$-fibred spaces over $S$, and denote the Yoneda embedding by $X \mapsto \h_S(X)$.
\end{defn}

\ssec{Nisnevich descent}

In this paragraph we discuss the property of Nisnevich descent for an $\Sm$-fibred space.
One very pleasant feature of the Nisnevich topology, compared to the étale topology, is that the sheaf condition can be described using finite limits (see \thmref{thm:excision=descent}).

\begin{defn}
Let $S$ be a spectral algebraic space and $X \in \Sm_{/S}$.
A \emph{Nisnevich square} over $X$ is a cartesian square of spectral algebraic spaces
  \begin{equation} \label{eq:Nisnevich square}
    \begin{tikzcd}
      W \arrow{r}\arrow{d}
        & V \arrow{d}{p}
      \\
      U \arrow{r}{j}
        & X
    \end{tikzcd}
  \end{equation}
where $j$ is an open immersion, $p$ is \'etale, and there exists a closed immersion $Z \hook X$ complementary to $j$ such that the induced morphism $p^{-1}(Z) \to Z$ is invertible.
\end{defn}

\begin{defn}
Let $\sF \in \Spc(S)$ be a fibred space over $S$.
We say that $\sF$ satisfies \emph{Nisnevich excision} if it is reduced, i.e. the space $\Gamma(\initial, \sF)$ is contractible, and for any Nisnevich square over $X$ of the form \eqref{eq:Nisnevich square}, the induced square of spaces
  \begin{equation*}
    \begin{tikzcd}
      \Gamma(X, \sF) \arrow{r}{j^*}\arrow{d}{p^*}
        & \Gamma(U, \sF) \arrow{d}
      \\
      \Gamma(V, \sF) \arrow{r}
        & \Gamma(W, \sF)
  \end{tikzcd}
  \end{equation*}
is cartesian.
\end{defn}

\begin{rem}\label{rem:Nisnevich filtered colimits}
Being defined by finite limits, the property of Nisnevich excision is stable under filtered colimits and small limits in $\Spc(S)$.
\end{rem}

\begin{defn}\label{defn:Nisnevich descent}
Let $\sF \in \Spc(S)$ be a fibred space over $S$.
We say that $\sF$ satisfies \emph{Nisnevich descent} if it is reduced and for any Nisnevich square \eqref{eq:Nisnevich square}, the canonical morphism of spaces
  \begin{equation*}
    \Gamma(X, \sF) \to \Tot \left(\Gamma(\Cech(\tilde{X}/X)_\bullet, \sF)\right)
  \end{equation*}
is invertible, where $\tilde{X} = U \coprod V$, the simplicial object $\Cech(\tilde{X}/X)_\bullet$ is the \v{C}ech nerve of the morphism $\tilde{X} \to X$, and ``$\Tot$'' denotes totalization of a cosimplicial diagram.
\end{defn}

\begin{constr}\label{constr:Nisnevich topology}
Consider the Grothendieck pretopology on $\Sm_{/S}$ generated by the following covering families: (a) the empty family, covering the empty scheme $\initial$; (b) for any $X \in \Sm_{/S}$ and for any Nisnevich square over $X$ of the form \eqref{eq:Nisnevich square}, the family $\{U \to X, V \to X\}$, covering $X$.
We call the associated Grothendieck topology the \emph{Nisnevich topology}.
Then $\sF \in \Spc(S)$ satisfies Nisnevich descent in the sense of \defref{defn:Nisnevich descent} if and only if it it is a sheaf with respect to the Nisnevich topology (in the sense of \cite{HTT}).
\end{constr}

\begin{thm}\label{thm:excision=descent}
Let $S$ be a spectral algebraic space and $\sF$ a $\Sm$-fibred space over $S$.
Then $\sF$ satisfies Nisnevich excision if and only if it satisfies Nisnevich descent.
\end{thm}

\thmref{thm:excision=descent} follows from a general result of Voevodsky \cite[Cor.~5.10]{VoevodskyCD} (cf. \cite[Thm.~3.2.5]{AsokHoyoisWendt}).

\begin{thm}[Voevodsky]\label{thm:abstract excision=descent}
Let $\bC$ be an \inftyCat admitting fibred products.
Let $\sE$ be a set of cartesian squares which is closed under isomorphism and satisfies the following properties:

\noindent{\em(a)}
The set $\sE$ is closed under base change.
More precisely, suppose that $Q$ is a cartesian square in $\bC$ of the form
  \begin{equation}\label{eq:abstract excision square}
    \begin{tikzcd}
      Q(1,1) \ar{r}\ar{d}
        & Q(0,1) \ar{d}
      \\
      Q(1,0) \ar{r}
        & Q(0,0)
    \end{tikzcd}
  \end{equation}
that belongs to $\sE$.
Then its base change along any morphism $c \to Q(0,0)$ in $\bC$ also belongs to $\sE$.

\noindent{\em(b)}
For every square $Q$ in $\sE$ of the form \eqref{eq:abstract excision square}, the lower horizontal arrow $Q(1,0) \to Q(0,0)$ is a monomorphism (i.e., its diagonal $\Delta : Q(1,0) \to Q(1,1) \fibprod_{Q(0,0)} Q(1,0)$ is invertible).

\noindent{\em(c)}
For every square $Q$ in $\sE$ of the form \eqref{eq:abstract excision square}, the right-hand vertical arrow $Q(0,1) \to Q(0,0)$ is $k$-truncated for some $k\ge 0$.

\noindent{\em(d)}
For every square $Q$ in $\sE$ of the form \eqref{eq:abstract excision square}, the induced square
  \begin{equation*}
    \begin{tikzcd}
      Q(1,1) \ar{r}\ar{d}{\Delta}
        & Q(0,1) \ar{d}{\Delta}
      \\
      Q(1,1) \fibprod_{Q(1,0)} Q(1,1) \ar{r}
        & Q(0,1) \fibprod_{Q(0,0)} Q(0,1)
    \end{tikzcd}
  \end{equation*}
also belongs to $\sE$.

\noindent
Then for any presheaf $\sF : (\bC)^\op \to \Spc$, the following two conditions are equivalent:

\noindent{\em(i)}
The presheaf $\sF$ sends every square in $\sE$ to a cartesian square of spaces.

\noindent{\em(ii)}
For any square $Q \in \sE$, write $\Cech(Q)_\bullet$ for the \v{C}ech nerve of the morphism $Q(1,0) \coprod Q(0,1) \to Q(0,0)$.
Then the canonical map of spaces
  \begin{equation*}
    \sF(Q(0,0)) \to \Tot(\sF(\Cech(Q)_\bullet))
  \end{equation*}
is invertible.
\end{thm}

\begin{notat}
Given a square $Q \in \sE$ of the form \eqref{eq:abstract excision square}, it will be useful to generalize the notation as follows: for each pair of integers $i,j \ge 0$, let $Q(i,j) $ denote the object
  \begin{equation*}
    Q(i,j) := Q(1,0)^{\times i} \fibprod_{Q(0,0)} Q(0,1)^{\times j}
  \end{equation*}
in $\bC$, where $Q(1,0)^{\times i}$ denotes the $i$-fold fibred product of $Q(1,0)$ with itself over $Q(0,0)$ (and similarly for $Q(0,1)^{\times j}$).
\end{notat}

\begin{proof}[Proof of \thmref{thm:abstract excision=descent}]
The proof is essentially the same as in the case where $\bC$ is a $1$-category, but we reproduce it here for the reader's convenience.
Given a square $Q \in \sE$, let $K_Q$ denote the colimit of the diagram $\h(Q(0,1)) \gets \h(Q(1,1)) \to \h(Q(1,0))$ (formed in the \inftyCat of presheaves), and let $\sK_\sE$ denote the set of canonical morphisms $k_Q : K_Q \to \h(Q(0,0))$ for all $Q \in \sE$.
Note that a presheaf $\sF$ satisfies condition (i) if and only if it is $\sK_\sE$-local.
Similarly, let $C_Q$ denote the geometric realization of the \v{C}ech nerve $\Cech(Q)_\bullet$, and $\sC_\sE$ the set of canonical morphisms $c_Q : \Cech(Q)_\bullet \to \h(Q(0,0))$ for all $Q \in \sE$.
Then a presheaf $\sF$ satisfies condition (ii) if and only if it is $\sC_\sE$-local.
For any $Q \in \sE$ as in \eqref{eq:abstract excision square}, form the cartesian square of presheaves
  \begin{equation}\label{eq:Voevodsky's magic square}
    \begin{tikzcd}
      K_Q \fibprod_{\h(Q(0,0))} C_Q \ar{r}{p(Q)}\ar{d}{q(Q)}
        & K_Q \ar{d}{k_Q}
      \\
      C_Q \ar{r}{c_Q}
        & \h(Q(0,0)).
    \end{tikzcd}
  \end{equation}
We will show that (1) the morphism $p(Q)$ is invertible, and that (2) $q(Q)$ is both a $\sK_\sE$-local equivalence and a $\sC_\sE$-local equivalence.
Since any class of local equivalences is strongly saturated \cite[Lem.~5.5.4.11]{HTT} and in particular satisfies the two-of-three property, it will follow that the classes of $\sK_\sE$-local and $\sC_\sE$-local equivalences coincide, and therefore that conditions (i) and (ii) are equivalent.

For (1), it suffices by universality of colimits to show that $c_Q$ becomes invertible after base change along any of the morphisms $Q(0,1) \to Q(1,1)$, $Q(1,0) \to Q(0,0)$, or $Q(1,1) \to Q(0,0)$.
Since the morphism $Q(0,1) \coprod Q(1,0) \to Q(0,0)$ splits after any of these base changes, it follows that the augmented simplicial object $\Cech(Q)_\bullet \to \h(Q(1,1))$ also becomes split after any of these base changes.

For (2), write $Q_{ij}$ for the base change of the square $Q$ along $Q(i,j) \to Q(0,0)$, for $i,j\ge 0$.
By universality of colimits, it will suffice to show that each $k_{Q_{ij}}$ is a $\sK_\sE$-local and $\sC_\sE$-local equivalence (for $i+j \ge 1$).
The former claim follows from assumption (a).
For $i \ge 1$, assumption (b) implies that the lower horizontal arrow in the square $Q_{ij}$ is invertible; in this case it is clear that the morphism $k_{Q_{ij}}$ is invertible.
Therefore we may set $i=0$ and consider the squares $Q_{0j}$ for $j\ge 1$; it will suffice to show that $k_{Q_{ij}}$ is invertible for sufficiently large $j$.
Note that in the commutative diagram
  \begin{equation*}
    \begin{tikzcd}
      Q(1,j)\ar{r}\ar{d}{\Delta}
        & Q(0,j) \ar{d}{\Delta}
      \\
      Q(1,j+1)\ar{r}\ar{d}
        & Q(0,j+1)\ar{d}
      \\
      Q(1,j) \ar{r}
        & Q(0,j)
    \end{tikzcd}
  \end{equation*}
both squares are cartesian, the vertical composites are identities, and the lower square is canonically identified with $Q_{0j}$.
Since the class of $\sC_\sE$-local equivalences is closed under retracts and cobase change, it will suffice to show that $k_{Q'}$ is a $\sC_\sE$-local equivalence, where $Q'$ denotes the upper square.
Note that by assumptions (a) and (d), the square $Q'$ belongs to $\sE$.
By assumption (b) its lower horizontal arrow is a monomorphism, and by (c) its right-hand vertical arrow is $(k-1)$-truncated (where $k$ is such that $Q(0,1)\to Q(0,0)$ is $k$-truncated).
Therefore we may replace $Q$ by $Q'$ and assume that the vertical arrow $Q(1,0) \to Q(0,0)$ is $(k-1)$-truncated.
Repeating the above argument recursively we eventually reduce to the case where both horizontal and vertical legs of the square $Q$ are $(-1)$-truncated (= monomorphisms).
For such $Q$, observe that in each of the squares $Q_{ij}$ ($i+j \ge 1$), one of the legs is invertible.
Then it is obvious that $k_{Q_{ij}}$ is invertible, so that $q(Q)$ is invertible.
Then the square \eqref{eq:Voevodsky's magic square} shows that $k_Q$ is a $\sC_\sE$-local equivalence.
\end{proof}

\begin{proof}[Proof of \thmref{thm:excision=descent}]
Apply \thmref{thm:abstract excision=descent} to the set of Nisnevich squares.
It is easy to see that the assumptions hold (recall that étale morphisms are $0$-truncated).
\end{proof}

\begin{rem}\label{rem:Lurie Nis topology}
From \cite[Thm.~3.7.5.1]{SAG-20180204} and \thmref{thm:excision=descent} it follows that the topology defined in \constrref{constr:Nisnevich topology} coincides with Lurie's version of the Nisnevich topology constructed in \sectsign~3.7.4 of \emph{loc. cit}.
\end{rem}

\begin{rem}\label{rem:hyper}
Note that in our definition of the Nisnevich descent property we consider only \v{C}ech covers as opposed to arbitrary hypercovers (see \cite[\sectsign~6.5.4]{HTT}).
If $S$ is noetherian and of finite dimension, then there is no difference \cite[Cor. 3.7.7.3]{SAG-20180204}.
\end{rem}

\begin{constr}
Let $\Spc_\Nis(S)$ denote the full subcategory of $\Spc(S)$ spanned by Nisnevich excisive fibred spaces.
By \thmref{thm:excision=descent} this is an exact left localization, and we denote the localization functor by $\sF \mapsto \LNis(\sF)$.
We say that a morphism in $\Spc(S)$ is a \emph{Nisnevich-local equivalence} if it induces an isomorphism after Nisnevich localization.
\end{constr}

\begin{exam}\label{exam:K-theory}
Given a spectral algebraic space $X$, let $\K(X)$ denote the nonconnective algebraic K-theory spectrum of the stable \inftyCat of perfect complexes on $X$.
Then the assignment $X \mapsto \Omega^\infty \K(X)$, viewed as an $\Sm$-fibred space over a spectral algebraic space $S$, is Nisnevich excisive.
This follows from compact generation of quasi-coherent sheaves on spectral algebraic spaces\footnotemark~ as in \cite[Thm.~9.6.1.1, Cor.~9.6.3.2]{SAG-20180204} and \cite[Prop.~6.9, Thm.~6.11]{AntieauGepnerBrauer} (see e.g. \cite[Prop.~A.13]{ClausenMathewNaumannNoelDescent}).
\footnotetext{Recall that for us, all spectral algebraic spaces are implicitly quasi-compact and quasi-separated.}
\end{exam}

\begin{prop} \label{prop:Spc_Nis generated by affines}
For any spectral algebraic space $S$, the \inftyCat $\Spc_\Nis(S)$ is generated under sifted colimits by objects of the form $\h_S(X)$, where $X \in \Sm_{/S}$ is affine and $X \to S$ factors through an étale morphism to a spectral affine space $\A^n_S$, for some $n\ge 0$.
\end{prop}

\begin{proof}[Proof of \propref{prop:Spc_Nis generated by affines}]
Say that $X \in \Sm_{/S}$ is \emph{good} if it admits an étale $S$-morphism to $\A^n_S$ for some $n\ge 0$.
Let $\bC$ denote the full subcategory of $\Spc_\Nis(S)$ generated under sifted colimits by representables of the form $\h_S(X)$ with $X$ affine and good.
From \cite[Lem.~5.5.8.14]{HTT} it follows that $\Spc_\Nis(S)$ is generated under sifted colimits by the representables, so it will suffice to show that every representable space $\h_S(X)$, $X \in \Sm_{/S}$, belongs to $\bC$.
Using \remref{rem:smooth structure} we may write $\h_S(X)$ as the colimit of a simplicial diagram where each term is good (namely, take the \v{C}ech nerve of a suitable Nisnevich covering family of $X$).
We may assume therefore that $X$ is good.

Choose a \emph{scallop decomposition}\footnotemark~$\initial = U_0 \hook U_1 \hook \cdots \hook U_n = X$ \cite[Thm. 3.5.2.1]{SAG-20180204}, i.e. a sequence of open immersions such that for each $1\le i \le n$ there exists an affine $V_i$ fitting in a commutative square
  \begin{equation*}
    \begin{tikzcd}
      W_i \ar[hookrightarrow]{r}\ar{d}
        & V_i \ar{d}
      \\
      U_{i-1} \ar[hookrightarrow]{r}
        & U_i
    \end{tikzcd}
  \end{equation*}
which is cartesian and cocartesian (in the \inftyCat of spectral algebraic spaces), with $V_i \to U_i$ étale.
\footnotetext{If $X$ admits a Zariski covering by affines, then one can take $V_i \to U_i$ to be open immersions, with $U_i = V_1 \cup V_2 \cup \cdots \cup V_i$.}
For each $i$ the morphisms $U_{i-1} \hook U_i$, $V_i \to U_i$ generate a Nisnevich covering, so $\h_S(U_i)$ is the colimit of the \v{C}ech nerve of $U_{i-1} \coprod V_i \to U_i$.
Moreover, every term of this simplicial diagram is separated, because $U_{i-1} \fibprod_{U_i} V_i = W_i$ is separated (being an open subspace in $V_i$).
By induction we may therefore assume that $X$ is separated and good.
For this case, choose again a scallop decomposition as above.
Since each $U_i$ is now separated and good, the cartesian squares
  \begin{equation*}
    \begin{tikzcd}
      W_i = U_{i-1} \fibprod_{U_i} V_i \ar{r}\ar{d}
        & U_{i-1} \times V_i \ar{d}
      \\
      U_i \ar{r}{\Delta}
        & U_i\times U_i
    \end{tikzcd}
  \end{equation*}
show that the $W_i$ are \emph{affine}.
The \v{C}ech nerve of $U_{i-1} \coprod V_i \to U_i$ is therefore a simplicial diagram with all terms affine and good, so we conclude by induction.
\end{proof}

\begin{prop}\label{prop:Nis sheaves on Sm^sch}
Let $S$ be a spectral algebraic space.
Denote by $\Sm^{\mrm{sch}}_{/S}$, resp. $\Sm^{\aff}_{/S}$, the full subcategory of $\Sm_{/S}$ spanned by smooth spectral \emph{schemes}, resp. \emph{affine} smooth spectral schemes, over $S$.

\noindent{\em(i)}
If $S$ is a spectral scheme, then restriction along the inclusion $\Sm^{\mrm{sch}}_{/S} \hook \Sm_{/S}$ induces an equivalence on \inftyCats of Nisnevich sheaves.
In particular, every Nisnevich sheaf on $\Sm_{/S}$ is a right Kan extension of its restriction to $\Sm^{\mrm{sch}}_{/S}$.

\noindent{\em(ii)}
If $S$ is affine, then restriction along the inclusion $\Sm^{\aff}_{/S} \hook \Sm_{/S}$ induces an equivalence on \inftyCats of Nisnevich sheaves.
In particular, every Nisnevich sheaf on $\Sm_{/S}$ is a right Kan extension of its restriction to $\Sm^{\aff}_{/S}$.
\end{prop}

\begin{proof}
For claim (i), let $\iota$ denote the inclusion functor and $\iota^*$ the functor of restriction of presheaves along $\iota$.
This admits fully faithful left and right adjoints $\iota_!$ and $\iota_*$, given respectively by left and right Kan extension.
Since $\iota$ is topologically continuous (preserves Nisnevich covering families), the functor $\iota^*$ preserves Nisnevich sheaves.
It is also topologically cocontinuous (see before \defref{def:quasi-cocontinuous functor}), so $\iota_*$ preserves Nisnevich sheaves.
Therefore at the level of Nisnevich sheaves the functor $\iota^*$ is left adjoint to $\iota_*$ and right adjoint to $\LNis \iota_!$.
It follows formally that $\LNis \iota_!$ is also fully faithful.
Since its essential image is generated under colimits by objects of the form $\LNis \iota_!(\h_S(X)) \simeq \LNis \h_S(X) \simeq \h_S(X)$, for $X \in \Sm^{\mrm{sch}}_{/S}$, it follows from \propref{prop:Spc_Nis generated by affines} that it is essentially surjective.
The proof of claim (ii) is the same.
\end{proof}

\ssec{\texorpdfstring{$\A^1$}{A1}-homotopy invariance}

\begin{defn}
Let $S$ be a spectral algebraic space and $\sF \in \Spc(S)$ a fibred space over $S$.
We say that $\sF \in \Spc(S)$ satisfies \emph{$\A^1$-homotopy invariance} if for every $X \in \Sm_{/S}$, the canonical map of spaces
  \begin{equation*}
    p^* : \Gamma(X, \sF) \to \Gamma(\A^1_X, \sF)
  \end{equation*}
is invertible, where $p : \A^1_X \to X$ is the projection of the spectral affine line over $X$ (\examref{exam:vector bundles}).
Let $\Spc_\htp(S)$ denote the full subcategory of $\Spc(S)$ spanned by $\A^1$-homotopy invariant fibred spaces.
\end{defn}

\begin{rem}\label{rem:htp colimits}
Note that the full subcategory $\Spc_\htp(S) \subset \Spc(S)$ is stable under small colimits and limits.
\end{rem}

\begin{defn}
Note that $\Spc_\htp(S)$ can be described as the (accessible) left localization of $\Spc(S)$ at the set of canonical projections $\A^1_X \to X$ for $X \in \Sm_{/S}$.
Since this set is essentially small, there is a localization functor $\sF \mapsto \Lhtp(\sF)$.
We say that a morphism in $\Spc(S)$ is an \emph{$\A^1$-local equivalence} if it induces an isomorphism in $\Spc_\htp(S)$ after $\A^1$-localization.
\end{defn}

\begin{exam}
The $\Sm$-fibred space $X \mapsto \Omega^\infty \K(X)$ of \examref{exam:K-theory} is rarely $\A^1$-homotopy invariant, and its $\A^1$-localization no longer satisfies Nisnevich descent.
There is however a variant of K-theory, studied in \cite{CisinskiKhanKH}, which is both $\A^1$-invariant and Nisnevich excisive.
\end{exam}

\begin{rem}\label{rem:Lhtp}
The fact that $\A^1$-projections are stable under base change implies that the $\A^1$-localization functor commutes with finite products, and in fact admits the following description: for any fibred space $\sF \in \Spc(S)$, the space of sections over any $X \in \Sm_{/S}$ is computed by a sifted colimit:
  \begin{equation} \label{eq:formula for Lhtp}
    \Gamma(X, \Lhtp(\sF)) \simeq \colim_{(Y \to X) \in (\bA_X)^\op} \Gamma(Y, \sF),
  \end{equation}
where $\bA_X$ is the full subcategory of $\Sm_{/X}$ spanned by composites of $\A^1$-projections.
Moreover, the \inftyCat $\Spc_{\htp}(S)$ has universality of colimits.
See \cite[Prop.~3.4]{HoyoisEquivariant}.
\end{rem}

\begin{defn}
Let $\bS$ denote the sphere spectrum and $\bS\{T\}$ the free $\Einfty$-algebra on one generator $T$ (in degree zero).
The two morphisms $\bS\{T\} \to \bS$ sending $T$ to $0$ and $1$, respectively, induce, for any $X \in \Sm_{/S}$, two sections $i_0$ and $i_1$ of the projection $p : \A^1_X \to X$.
Given two morphisms $\varphi_0,\varphi_1 : \sF \rightrightarrows \sG$ in $\Spc(S)$, an \emph{elementary $\A^1$-homotopy} from $\varphi_0$ to $\varphi_1$ is a morphism $\h_S(\A^1_S) \times \sF \to \sG$ whose restrictions to $\h_S(S) \times \sF = \sF$ along $i_0$ and $i_1$ are isomorphic to $\varphi_0$ and $\varphi_1$, respectively.
We say that $\varphi_0$ and $\varphi_1$ are \emph{$\A^1$-homotopic} if there exists a sequence of elementary $\A^1$-homotopies connecting them; in this case the induced morphisms $\Lhtp(\sF) \rightrightarrows \Lhtp(\sG)$ coincide.
A morphism $\varphi: \sF \to \sG$ in $\Spc(S)$ is called a \emph{strict $\A^1$-homotopy equivalence} if there exists a morphism $\psi : \sG \to \sF$ such that the composites $\varphi \circ \psi$ and $\psi \circ \varphi$ are $\A^1$-homotopic to the identities.
Any strict $\A^1$-homotopy equivalence is an $\A^1$-local equivalence.
\end{defn}

\ssec{Motivic spaces}

\begin{defn}\label{defn:MotSpc}
A \emph{motivic space} over $S$ is a $\Sm$-fibred space $\sF \in \Spc(S)$ that is Nisnevich-local and $\A^1$-local.
We write $\MotSpc(S)$ for the full subcategory of $\Spc(S)$ spanned by motivic spaces.
This is an accessible left localization, and we write $\sF \mapsto \Lmot(\sF)$ for the localization functor.
We say that a morphism in $\Spc(S)$ is a \emph{motivic equivalence} if it induces an isomorphism in $\MotSpc(S)$ after application of $\Lmot$.
We write $\hmot[S]{X} := \Lmot \h_S(X)$ for the motivic space represented by an object $X \in \Sm_{/S}$.
\end{defn}

\begin{rem}\label{rem:MotSpc categorical}
The \inftyCat $\MotSpc(S)$ has universality of colimits (since $\Spc_\Nis(S)$ and $\Spc_\htp(S)$ do).
Similarly, the functor $\Lmot$ commutes with finite products (since $\LNis$ and $\Lhtp$ do).
By adjunction it follows that $\MotSpc(S)$ is cartesian closed: for any $\sF \in \MotSpc(S)$, $\sG \in \Spc(S)$, the internal hom object $\uHom(\sG, \sF) \in \Spc(S)$ is a motivic space.
\end{rem}

\begin{rem}\label{rem:Lmot}
Since the conditions of Nisnevich and $\A^1$-locality are each stable under filtered colimits (Remarks~\ref{rem:Nisnevich filtered colimits} and \ref{rem:htp colimits}), it follows that motivic localization can be described as the transfinite composite
  \begin{equation} \label{eq:description of Lmot}
    \Lmot(\sF) = \colim_{n \ge 0} (\Lhtp \circ \LNis)^{\circ n}(\sF)
  \end{equation}
for any $\sF \in \Spc(S)$.
\end{rem}

By \propref{prop:Spc_Nis generated by affines} we get:

\begin{prop} \label{prop:MotSpc generated by affines}
The \inftyCat $\MotSpc(S)$ is generated under sifted colimits by objects of the form $\hmot[S]{X}$, where $X\in\Sm_{/S}$ is affine and $X \to S$ factors through an étale morphism to a spectral affine space $\A^n_S$, for some $n\ge 0$.
\end{prop}

\begin{cor}\label{cor:MotSpc on Sm^aff}
Let $S$ be a spectral algebraic space.
Consider the full subcategories $\Sm^{\mrm{sch}}_{/S}$ and $\Sm^{\aff}_{/S}$ of $\Sm_{/S}$ as defined in \propref{prop:Nis sheaves on Sm^sch}.

\noindent{\em(i)}
If $S$ is a spectral scheme, then restriction along the inclusion $\Sm^{\mrm{sch}}_{/S} \hook \Sm_{/S}$ induces an equivalence on \inftyCats of $\A^1$-homotopy invariant Nisnevich sheaves.

\noindent{\em(i)}
If $S$ is an affine spectral scheme, then restriction along the inclusion $\Sm^{\aff}_{/S} \hook \Sm_{/S}$ induces an equivalence on \inftyCats of $\A^1$-homotopy invariant Nisnevich sheaves.
\end{cor}

\begin{proof}
Let $\iota$ denote either inclusion functor.
The claim will follow by repeating the proof of \propref{prop:Nis sheaves on Sm^sch} and using the following assertions:

($*$) The restriction functor $\iota^*$ preserves $\A^1$-invariant spaces, as does its right adjoint $\iota_*$.

The first claim follows immediately from the fact that $\iota$ preserves $\A^1$-projections (so that the left Kan extension functor $\iota_!$ preserves $\A^1$-local equivalences).
The second claim is equivalent by adjunction to the assertion that $\iota^*$ preserves $\A^1$-local equivalences.
For this it will suffice to show that, for any $X\in\Sm_{/S}$, the canonical morphism
  \begin{equation*}
    \iota^*\h_S(X \times \A^1) \to \iota^*\h_S(X)
  \end{equation*}
is an $\A^1$-local equivalence.
By universality of colimits it suffices to show that, for any $Y\in\Sm_{/S}$ and any morphism $\varphi : \h_S(Y) \to \iota^*\h_S(X)$, the base change
  \begin{equation*}
    \iota^*\h_S(X \times \A^1) \fibprod_{\iota^*\h_S(X)} \h_S(Y) \to \h_S(Y)
  \end{equation*}
is an $\A^1$-local equivalence.
Since $\varphi$ factors as $\h_S(Y) \to \iota^*\iota_!\h_S(Y) \simeq \iota^*\h_S(Y) \to \iota^*\h_S(X)$, the morphism in question is a base change of the morphism
  \begin{equation*}
    \iota^*\h_S(X \times \A^1) \fibprod_{\iota^*\h_S(X)} \iota^*\h_S(Y) \to \iota^*\h_S(Y).
  \end{equation*}
Since $\iota^*$ commutes with limits, this is identified with the canonical morphism $\h_S(Y \times \A^1) \to \h_S(Y)$.
This is obviously an $\A^1$-local equivalence, so the claim follows.
\end{proof}

Let $S$ be a classical scheme.
Then there is a parallel variant $\MotSpc_\cl(S)$ of the construction $\MotSpc(S)$ using the site of smooth \emph{classical} schemes, imposing Nisnevich descent and homotopy invariance with respect to the \emph{classical} affine line; see e.g. \cite[Appendix~C]{HoyoisLefschetz} (where it is denoted $\MotSpc(S)$).
This agrees with the original construction of Morel--Voevodsky \cite{MorelVoevodsky} when the latter is defined (that is, when $S$ is noetherian and of finite dimension).
If we assume that $S$ is of characteristic zero, then it also agrees with the spectral version $\MotSpc(S)$:

\begin{prop}\label{prop:char 0}
Let $S$ be a classical scheme of characteristic zero.
Then there is a canonical equivalence of \inftyCats
  \begin{equation*}
    \MotSpc(S) \simeq \MotSpc_\cl(S).
  \end{equation*}
\end{prop}

\begin{proof}
This follows from \corref{cor:MotSpc on Sm^aff} and the fact that, if $S$ is of characteristic zero, there is a canonical equivalence between $\Sm^{\mrm{sch}}_{/S}$ and the category of smooth classical $S$-schemes, which preserves and detects Nisnevich covers, and sends the spectral affine line to the classical one.
\end{proof}

\begin{warn}\label{warn:char p}
The characteristic zero hypothesis in \propref{prop:char 0} is necessary in the proof: for example, if $S = \Spec(\bF_p)$, then the site $\Sm^{\mrm{sch}}_{/S}$ is not equivalent to the site of smooth classical $S$-schemes.
Indeed the classical affine line $\Spec(\bF_p[T])$ is not smooth over $S$ when viewed as a spectral scheme.
On the other hand, the site $\Sm^{\mrm{sch}}_{/S}$ contains objects like the spectral affine line $\A^1_S = \Spec(\bF_p\{T\})$, which is not a classical scheme.
See \cite[Prop.~2.4.1.5]{HAG2}.
\end{warn}

We conclude this subsection by introducing the pointed variant of $\MotSpc(S)$:

\begin{constr}\label{constr:pointed MotSpc}
Given a spectral algebraic space $S$, write $\MotSpc(S)_\bullet$ for the \inftyCat of pointed objects in $\MotSpc(S)$.
The forgetful functor $\MotSpc(S)_\bullet\to\MotSpc(S)$ admits a left adjoint $\sF \mapsto \sF_+$ which freely adjoins a base point.
The \inftyCat $\MotSpc(S)_\bullet$ is equivalent to the full subcategory of the \inftyCat $\Spc(S)_\bullet$ of fibred pointed spaces $\sF$ whose underlying fibred space is motivic.
It is an accessible left localization of $\Spc(S)_\bullet$ and the localization functor satisfies $\Lmot(\sF_+) \simeq \Lmot(\sF)_+$ for any $\sF \in \MotSpc(S)$. 
The cartesian monoidal structure on $\MotSpc(S)$ extends uniquely to a symmetric monoidal structure on $\MotSpc(S)_\bullet$ with the property that the functor $\sF \mapsto \sF_+$ is symmetric monoidal \cite[Cor.~2.32]{RobaloNoncommutative}.
We write $\wedge$ for the monoidal product; the monoidal unit is the object $(\pt_S)_+ \in \MotSpc(S)_\bullet$.
It follows from \propref{prop:MotSpc generated by affines} that $\MotSpc(S)_\bullet$ is generated under sifted colimits by objects of the form $\hmot[S]{X}_+$, for all affine $X \in \Sm_{/S}$ which admit an étale $S$-morphism to $\A^n_S$ for some $n\ge 0$.
\end{constr}

\ssec{Functoriality}

We now discuss the basic functorialities that the system of categories $\MotSpc(S)$ enjoys as $S$ varies.
For any morphism $f : T \to S$, we will define a pair of adjoint functors $(f^*_{\MotSpc}, f_*^{\MotSpc})$.
If $f$ happens to be smooth, there will be a further adjunction $(f_\sharp^{\MotSpc}, f^*_{\MotSpc})$.
When there is no risk of confusion we will usually omit the decorations $^\MotSpc$ and $_\MotSpc$.

\begin{constr}
Let $f : T \to S$ be a morphism of spectral algebraic spaces.
Restriction along the base change functor $\Sm_{/S} \to \Sm_{/T}$ defines a canonical functor
  \begin{equation*}
    f_*^{\Spc} : \Spc(T) \to \Spc(S).
  \end{equation*}
It admits a left adjoint $f^*_{\Spc}$ which is uniquely characterized by commutativity with small colimits and the formula $f^*_{\Spc}(\hspc[S]{X}) \simeq \hspc[T]{X \fibprod_S T}$ for $X \in \Sm_{/S}$.
\end{constr}

\begin{constr}
The base change functor $\Sm_{/S} \to \Sm_{/T}$ preserves Nisnevich covering families and $\A^1$-projections, so the functor $f^*_{\Spc}$ preserves motivic equivalences.
By adjunction its right adjoint $f_*^{\Spc}$ preserves motivic spaces and induces a functor (``direct image'')
  \begin{equation*}
    f_*^{\MotSpc} : \MotSpc(T) \to \MotSpc(S).
  \end{equation*}
This is right adjoint to the functor $f^*_{\MotSpc} = \Lmot f^*_{\Spc}$ (``inverse image''), characterized uniquely by commutativity with colimits and the formula $f^*_\MotSpc(\hmot[S]{X}) \simeq \hmot[T]{X \fibprod_S T}$ for $X \in \Sm_{/S}$.
\end{constr}

\begin{rem}
By \remref{rem:MotSpc categorical} it follows that $f^*_\MotSpc$ commutes with finite products.
\end{rem}

When the morphism $f$ is smooth, the inverse image functor $f^*_\MotSpc$ also admits a \emph{left} adjoint $f_\sharp$.
When $f$ is an open immersion, this is the functor of ``extension by zero''.
More generally, when $f$ is étale, it is the functor of ``compactly support direct image''.

\begin{constr}
Let $p : T \to S$ be a smooth morphism of spectral algebraic spaces.
Then the base change functor admits a right adjoint $\Sm_{/T} \to \Sm_{/S}$, the forgetful functor $(X \to T) \mapsto (X \to T \xrightarrow{p} S)$.
It follows that the functor $p^*_{\Spc}$ coincides with restriction along the forgetful functor, and admits a left adjoint
  \begin{equation*}
    p_\sharp^{\Spc} : \Spc(T) \to \Spc(S),
  \end{equation*}
which is uniquely characterized by commutativity with colimits and the formula $p_\sharp^{\Spc}(\hspc[T]{X}) \simeq \hspc[S]{X}$ for $X \in\Sm_{/T}$.
\end{constr}

\begin{constr}
Since the forgetful functor $\Sm_{/T} \to \Sm_{/S}$ preserves Nisnevich covering families and $\A^1$-projections, it follows that $p_\sharp^{\Spc}$ preserves motivic equivalences.
In particular its right adjoint $p^*_{\Spc}$ preserves motivic spaces, and induces a functor
  \begin{equation*}
    p^*_{\MotSpc} : \MotSpc(S) \to \MotSpc(T).
  \end{equation*}
This is right adjoint to the functor $p_\sharp^{\MotSpc}(\sF) \simeq \Lmot(p_\sharp^{\Spc}(\sF))$, characterized by commutativity with colimits and the formula $p_\sharp^\MotSpc(\hmot[T]{X}) \simeq \hmot[S]{X}$ for $X \in \Sm_{/T}$.
\end{constr}

\begin{rem}
Note that $p_\sharp^\MotSpc$ commutes with binary products if and only if $p$ is a monomorphism, hence if and only if $p$ is an open immersion.
\end{rem}

\begin{prop}[Nisnevich separation] \label{prop:Nisnevich separation of MotSpc}
Let $(p_\alpha : S_\alpha \to S)_\alpha$ be a Nisnevich covering family of spectral algebraic spaces.
Then the functors $(p_\alpha)^* : \MotSpc(S) \to \MotSpc(S_\alpha)$ form a conservative family.
\end{prop}

\begin{proof}
Let $\varphi : \sF_1 \to \sF_2$ be a morphism in $\MotSpc(S)$ and assume that $(p_\alpha)^*(\varphi)$ is invertible for each $\alpha$.
It suffices to show that the map $\Gamma(X, \sF_1) \to \Gamma(X, \sF_2)$ is invertible for all $X \in \Sm_{/S}$.
Passing to the \v{C}ech nerve of the covering family, we may assume that $X \in \Sm_{/S_\alpha}$ for some $\alpha$.
Then the claim follows from the assumption because we have by adjunction natural isomorphisms $\Gamma(X_\alpha, \sF_i) \simeq \Gamma(X, (p_\alpha)_\sharp (p_\alpha)^* \sF_i)$ for each $i$.
\end{proof}

\begin{constr}
Suppose we have a cartesian square
  \begin{equation*}
    \begin{tikzcd}
      T' \arrow{r}{f'}\arrow{d}{p'}
        & S' \arrow{d}{p}
      \\
      T \arrow{r}{f}
        & S
    \end{tikzcd}
  \end{equation*}
of spectral algebraic spaces.
If $p$ is smooth, then the square
  \begin{equation*}
    \begin{tikzcd}
      \MotSpc(S')\ar{r}{(f')^*}\ar{d}{p_\sharp}
        & \MotSpc(T')\ar{d}{(p')_\sharp}
      \\
      \MotSpc(S)\ar{r}{f^*}
        & \MotSpc(T)
    \end{tikzcd}
  \end{equation*}
commutes up to the canonical natural transformation
  \begin{equation*}
    (p')_\sharp (f')^* \xrightarrow{\mrm{unit}} (p')_\sharp (f')^* p^* p_\sharp \simeq (p')_\sharp (p')^* f^* p_\sharp \xrightarrow{\mrm{counit}} f^* p_\sharp.
  \end{equation*}
\end{constr}

\begin{prop}[Smooth base change formula] \label{prop:smooth base change}
Suppose given a cartesian square of spectral algebraic spaces as above, with $p$ smooth.
Then there are canonical invertible natural transformations
  \begin{align*}
    (p')_\sharp^{\MotSpc} (f')^*_{\MotSpc} \to f^*_{\MotSpc} p_\sharp^{\MotSpc},
    \\
    p^*_{\MotSpc} f_*^{\MotSpc} \to (f')_*^{\MotSpc} (p')^*_{\MotSpc},
  \end{align*}
of functors $\MotSpc(S') \to \MotSpc(T)$ and $\MotSpc(T) \to \MotSpc(S')$, respectively.
\end{prop}

\begin{proof}
The second transformation is obtained from the first by passage to right adjoints.
For the first, note that each of the functors involved commutes with colimits.
Therefore by \propref{prop:MotSpc generated by affines} it suffices to show that the transformation induces an isomorphism after evaluation at any object object $\hmot[S']{X'}$, with $X' \in \Sm_{/S'}$, which is obvious.
\end{proof}

\begin{cor} \label{cor:j_sharp and j_* fully faithful}
Let $j : U \hook S$ be an open immersion of spectral algebraic spaces.
Then the functors $j_\sharp^{\MotSpc}$ and $j_*^{\MotSpc}$ are fully faithful.
\end{cor}

\begin{proof}
Applying \propref{prop:smooth base change} to the square
  \[ \begin{tikzcd}
    U \arrow[equals]{r}\arrow[equals]{d}
      & U \arrow[hookrightarrow]{d}{j}
    \\
    U \arrow[hookrightarrow]{r}{j}
      & S,
  \end{tikzcd} \]
which is cartesian because $j$ is a monomorphism, we deduce that the natural transformations $\id \to j^*j_\sharp$ and $j^*j_* \to \id$ are invertible.
\end{proof}

\begin{cor} \label{cor:j^*i_* trivial}
Let $i : Z \hook S$ be a closed immersion of spectral algebraic spaces with quasi-compact open complement $j : U \hook S$.
Then the natural transformations
  \begin{align*}
    \initial_Z \to i^*_{\MotSpc}j_\sharp^{\MotSpc}, \\
    j^*_{\MotSpc}i_*^{\MotSpc} \to \pt_U,
  \end{align*}
are invertible.
\end{cor}

\begin{proof}
Apply \propref{prop:smooth base change} to the cartesian square
  \[ \begin{tikzcd}
    \initial \arrow[hookrightarrow]{r}\arrow[hookrightarrow]{d}
      & Z \arrow[hookrightarrow]{d}{i}
    \\
    U \arrow[hookrightarrow]{r}{j}
      & S.
  \end{tikzcd} \]
\end{proof}

\begin{constr}
Let $p : S' \to S$ be a smooth morphism of spectral algebraic spaces.
Given motivic spaces $\sF' \in \MotSpc(S')$, $\sF \in \MotSpc(S)$ and $\sG \in \MotSpc(S)$, we get a morphism
  \begin{equation*}
    \sF' \fibprod_{p^*_\MotSpc(\sG)} p^*_{\MotSpc}(\sF)
      \xrightarrow{\mrm{unit}} p^*_{\MotSpc} p_\sharp^{\MotSpc}(\sF') \fibprod_{p^*_\MotSpc(\sG)} p^*_{\MotSpc}(\sF)
      \simeq p^*_{\MotSpc}(p_\sharp^{\MotSpc}(\sF') \fibprod_{\sG} \sF),
  \end{equation*}
which corresponds by adjunction to a canonical morphism
  \begin{equation*}
    p_\sharp^{\MotSpc}(\sF' \fibprod_{p^*_{\MotSpc}(\sG)} p^*_{\MotSpc}(\sF)) \to p_\sharp^{\MotSpc}(\sF') \fibprod_{\sG} \sF.
  \end{equation*}
\end{constr}

\begin{prop}[Smooth projection formula]\label{prop:smooth projection formula}
Let $p : S' \to S$ be a smooth morphism of spectral algebraic spaces.
Given motivic spaces $\sF' \in \MotSpc(S')$, $\sF \in \MotSpc(S)$ and $\sG \in \MotSpc(S)$, we have canonical bifunctorial isomorphisms
  \begin{align*}
    p_\sharp^{\MotSpc}(\sF' \fibprod_{p^*_{\MotSpc}(\sG)} p^*_{\MotSpc}(\sF)) \isoto p_\sharp^{\MotSpc}(\sF') \fibprod_{\sG} \sF,\\
    p^*_{\MotSpc}\left(\uHom(\sF, \sG)\right) \isoto \uHom(p^*_{\MotSpc}(\sF), p^*_{\MotSpc}(\sG)),\\
    \uHom(p_\sharp^{\MotSpc}(\sF'), \sF) \isoto p_*^{\MotSpc}\uHom(\sF', p^*_{\MotSpc}(\sF)).
  \end{align*}
\end{prop}

The functorialities under discussion extend freely to the setting of pointed motivic spaces:

\begin{constr}
Given a spectral algebraic space $S$, let $\MotSpc(S)_\bullet$ denote the \inftyCat of pointed motivic spaces over $S$ (\constrref{constr:pointed MotSpc}).
For any morphism $f : T \to S$, the direct image functor $f_*^{\MotSpc}$ preserves terminal objects and therefore induces a functor $f_*^{\MotSpc_\bullet}$ (which commutes with passage to underlying motivic spaces).
Its left adjoint $f^*_{\MotSpc_\bullet}$ is uniquely characterized by commutativity with sifted colimits and the formula $f^*_{\MotSpc_\bullet}(\sF_+) \simeq f^*_{\MotSpc}(\sF)_+$ for any $\sF \in \MotSpc(S)$.
Similarly, for any smooth morphism $p : T \to S$, there is a functor $p_\sharp^{\MotSpc_\bullet}$ left adjoint to $p^*_{\MotSpc_\bullet}$ which is uniquely characterized by commutativity with sifted colimits and the formula $p_\sharp^{\MotSpc_\bullet}(\sF_+) \simeq p_\sharp^{\MotSpc}(\sF)_+$ for any $\sF \in \MotSpc(T)$.
The smooth base change and projection formulas (Propositions~\ref{prop:smooth base change} and \ref{prop:smooth projection formula}) have obvious pointed analogues that we leave to the reader to formulate.
\end{constr}

\section{Results}
\label{sec:results}

\ssec{Exactness of \texorpdfstring{$i_*$}{i lower star}}

Our first goal is to prove the following:

\begin{thm} \label{thm:i_* commutes with contractible colimits}
Let $i : Z \hook S$ be a closed immersion of spectral algebraic spaces.
Then the direct image functor $i_*^{\MotSpc} : \MotSpc(Z) \to \MotSpc(S)$ commutes with contractible colimits.
\end{thm}

In other words, $i_*^{\MotSpc}$ commutes with colimits of diagrams indexed by contractible\footnote{An essentially small \inftyCat is called \emph{contractible} if the \inftyGrpd obtained by formally adjoining inverses to all morphisms, is (weakly) contractible.} \inftyCats.
At the level of pointed spaces, we get:

\begin{cor} \label{cor:i_* admits right adjoint i^!}
The direct image functor $i_*^{\MotSpc_\bullet} : \MotSpc(Z)_\bullet \to \MotSpc(S)_\bullet$ commutes with small colimits.
\end{cor}

\begin{proof}
Note that it suffices to show that $i_*^{\MotSpc_\bullet}$ commutes with contractible colimits and preserves the initial object.\footnotemark
\footnotetext{Let $D$ be a diagram indexed on an \inftyCat $\bI$.
Then the initial object defines a cone over $D$, which we may view as another diagram $D'$ with the same colimit but which is indexed on a \emph{contractible} \inftyCat.}
The former condition follows directly from the unpointed statement (\thmref{thm:i_* commutes with contractible colimits}), and the fact that $i_*^{\MotSpc_\bullet}$ preserves the initial object (= terminal object) is obvious.
\end{proof}

\begin{rem}
It follows from \corref{cor:i_* admits right adjoint i^!} that the functor $i_*^{\MotSpc_\bullet}$ admits a right adjoint $i^!_{\MotSpc_\bullet}$, called the \emph{exceptional inverse image} functor.
A concrete description of the functor $i^!_{\MotSpc_\bullet}$ can be given using the localization theorem: see \remref{rem:i^! concrete}.
\end{rem}

The geometric input into the proof of \thmref{thm:i_* commutes with contractible colimits} is a spectral version of \cite[Prop.~18.1.1]{EGAIV4}:

\begin{prop}\label{prop:Nisnevich lifting}
Let $i : Z \hook S$ be a closed immersion of spectral algebraic spaces.
For any smooth (resp. \'etale) morphism $q : Y \to Z$, there exists, Nisnevich-locally on $Y$, a smooth (resp. \'etale) morphism $p: X \to S$, and a cartesian square
  \begin{equation*}
    \begin{tikzcd}
      Y \arrow{d}{q} \arrow{r}
        & X\arrow{d}{p} \\
      Z \arrow[hookrightarrow]{r}
        & S.
    \end{tikzcd}
  \end{equation*}
\end{prop}

\begin{proof}
The smooth case follows from the étale case by factoring $q : Y \to Z$ through an étale morphism to an affine space $\A^n_Z$ (which can always be done Nisnevich-locally).
Therefore it will suffice to show the étale case.
The question being Nisnevich-local, we may assume that $S$, $Z$ and $Y$ are affine; let $S = \Spec(R)$, $Z = \Spec(R')$ and $Y = \Spec(A')$.
The étale $R'$-algebra $A'$ induces an étale $\pi_0(R')$-algebra $A' \otimes_{R'} \pi_0(R') \simeq \pi_0(A')$.
Note that it will suffice to demonstrate the claim Zariski-locally on $\pi_0(A')$ (since any Zariski covering of $\pi_0(A')$ lifts uniquely to a Zariski covering of $A'$).
By \cite[Prop.~18.1.1]{EGAIV4} the $\pi_0(R')$-algebra $\pi_0(A')$ lifts, Zariski-locally on $\pi_0(A')$, to an étale $\pi_0(R)$-algebra $A_0$.
By \cite[Thm.~7.5.0.6]{HA-20170918}, the latter lifts uniquely to an étale $R$-algebra $A$.
\end{proof}

In order to deduce \thmref{thm:i_* commutes with contractible colimits} from \propref{prop:Nisnevich lifting}, we will need to make a small topos-theoretic digression.
Let $\bC$ and $\bD$ be essentially small \inftyCats, equipped with Grothendieck topologies $\tau_\bC$ and $\tau_\bD$, respectively.
Recall that a functor $u : \bC \to \bD$ is \emph{topologically cocontinuous} if the following condition is satisfied:

(COC)
For every $\tau_\bD$-covering sieve $R' \hook h(u(c))$, the sieve $R \hook h(c)$, generated by morphisms $c' \to c$ such that $h(u(c')) \to h(u(c))$ factors through $R'$, is $\tau_\bC$-covering.

\begin{defn} \label{def:quasi-cocontinuous functor}
Let $\bC$ and $\bD$ be essentially small \inftyCats, equipped with Grothendieck topologies $\tau_\bC$ and $\tau_\bD$, respectively.
Assume that $\bD$ admits an initial object $\initial_\bD$.
A functor $u : \bC \to \bD$ is \emph{topologically quasi-cocontinuous} if it satisfies the following condition:

\emph{($\text{COC'}$)}
For every $\tau_\bD$-covering sieve $R' \hook h(u(c))$, the sieve $R \hook h(c)$, generated by morphisms $c' \to c$ such that either $u(c')$ is initial or $h(u(c')) \to h(u(c))$ factors through $R' \hook h(u(c))$, is $\tau_\bC$-covering.
\end{defn}

Let $\PSh(\bD)$ denote the \inftyCat of presheaves of spaces on $\bC$, $\Sh_{\tau_\bD}(\bD)$ the full subcategory of $\PSh(\bD)$ spanned by $\tau_{\bD}$-sheaves.
We write $\sF \mapsto \L_{\tau_\bD}(\sF)$ for the (left-exact) localization functor.

\begin{lem} \label{lem:quasi-cocontinuous and contractible colimits}
With notation as in \defref{def:quasi-cocontinuous functor}, let $u : \bC \to \bD$ be a topologically quasi-cocontinuous functor.
Assume that the initial object $\initial_\bD$ is \emph{strict} in the sense that for any object $d\in\bD$, any morphism $d \to \initial_\bD$ is invertible.
Assume also that, for any object $d\in\bD$, the sieve $\initial_{\PSh(\bD)} \hook h(d)$ is $\tau_\bD$-covering if and only if $d$ is initial (where $\initial_{\PSh(\bD)}$ denotes the initial object of $\PSh(\bD)$).
Then we have:

\noindent{\em(i)}
The restriction functor $u^* : \PSh(\bD) \to \PSh(\bC)$ sends $\tau_\bD$-local equivalences between reduced\footnotemark~presheaves to $\tau_\bC$-local equivalences.
\footnotetext{We say that a presheaf $\sF$ is reduced if it sends the initial object to a contractible space.}

\noindent{\em(ii)}
The functor $\Sh_{\tau_\bD}(\bD) \to \Sh_{\tau_\bC}(\bC)$, given by the assignment $\sF \mapsto \L_{\tau_\bC}(u^*(\sF))$, commutes with contractible colimits.
\end{lem}

\begin{proof}
Let $\PSh_\red(\bD)$ the full subcategory of $\PSh(\bD)$ spanned by reduced presheaves.
This is a left localization, and the localization functor $\sF \mapsto \L_\red(\sF)$ has the effect of forcing $\L_\red(\sF)(\initial) \simeq \pt$ (while $\sF(d) \to \L_\red(\sF)(d)$ is an isomorphism whenever $d$ is not initial).
Note also that the \inftyCat $\PSh_\red(\bD)$ is the free completion of $\bD$ by contractible colimits.

Let $\sA$ denote the set of morphisms in $\PSh(\bD)$ containing all isomorphisms, the canonical morphism $e : \initial_{\PSh(\bD)} \hook h(\initial_\bD)$, and all $\tau_\bD$-covering sieves $R' \hook h(d)$ of a non-initial object $d \in \bD$.
Then the set of $\tau_\bD$-local equivalences in $\PSh(\bD)$ is the closure of $\sA$ under small colimits, cobase change, and the $2$-of-$3$ property.
It follows that the set of $\tau_\bD$-local equivalences in $\PSh_{\red}(\bD)$ is the closure of the set $\L_\red(\sA)$ under \emph{contractible} colimits, cobase change, and the $2$-of-$3$ property.
Therefore, for the first claim it will suffice to show that $u^*$ sends every morphism in $\L_\red(\sA)$ to a $\tau_\bC$-local equivalence (since the subcategory $\PSh_\red(\bD)$ is closed under contractible colimits).
This is clear for the morphism $\L_\red(e)$, as it is already invertible.

Now let $s : R' \hook h(d)$ be a $\tau_\bD$-covering sieve of a non-initial object $d\in\bD$.
Note that we have $\L_\red(s) = s$, as $h(d)$ is reduced and hence so is $R'$ (since $s$ is a monomorphism and $R'$ is by assumption non-empty).
Thus we need to show that $u^*(s)$ is a $\tau_\bC$-local equivalence.
By universality of colimits it is sufficient to show that, for every object $c$ of $\bC$ and every morphism $\varphi : h(c) \to u^* h(d)$, the base change
  \begin{equation*}
    u^* R' \fibprod_{u^* h(d)} h(c) \hook h(c)
  \end{equation*}
is a $\tau_\bC$-covering sieve.
By adjunction, $\varphi$ factors through the unit $h(c) \to u^* u_! h(c) = u^* h(u(c))$ and the morphism $u^*(\varphi^\flat) : u^* h(u(c)) \to u^* h(d)$, where $\varphi^\flat$ is the left transpose of $\varphi$.
The base change of $\varphi$ by $u^* h(u(c)) \to u^* h(d)$ is identified, since $u^*$ commutes with limits, with the canonical morphism
  \begin{equation*}
    u^*(R' \fibprod_{h(d)} h(u(c))) \to u^* h(u(c)).
  \end{equation*}
Since the sieve $R' \fibprod_{h(d)} h(u(c)) \hook h(u(c))$ is $\tau_\bD$-covering, as the base change of a $\tau_\bD$-covering sieve, the conclusion now follows from the condition (COC').

The second assertion is a formal consequence of the first, using the fact that every $\tau_\bD$-sheaf is reduced (by assumption).
\end{proof}

\begin{lem}\label{lem:quasi-cocontinuous example}
Let $i : Z \hook S$ be a closed immersion of spectral algebraic spaces.
Then the base change functor $\Sm_{/S} \to \Sm_{/Z}$ is topologically quasi-cocontinuous (with respect to the Nisnevich topology).
\end{lem}

\begin{proof}
Unravelling the definition, this amounts to the following assertion:

$(\ast)$ For any $X\in\Sm_{/S}$ and any Nisnevich covering sieve $R'$ of $X \fibprod_S Z$, let $R \hook \h_S(X)$ denote the sieve generated by morphisms $X' \to X$ such that either (i) the empty sieve on $X'\fibprod_S Z$ is Nisnevich-covering, or (ii) $X' \fibprod_S Z \to X \fibprod_S Z$ factors through $R'$.
Then $R \hook \h_S(X)$ is Nisnevich-covering.

This follows directly from \propref{prop:Nisnevich lifting}, which says that \'etale morphisms can be lifted (Nisnevich-locally) along $i$.
\end{proof}

\begin{proof}[Proof of \thmref{thm:i_* commutes with contractible colimits}]
Follows from Lemmas~\ref{lem:quasi-cocontinuous and contractible colimits} and \ref{lem:quasi-cocontinuous example}.
\end{proof}

\ssec{The localization theorem}

We now state the main result of this paper, and explain some of its immediate consequences.

\begin{constr}
Let $i : Z \hook S$ be a closed immersion of spectral algebraic spaces with quasi-compact open complement $j : U \hook S$.
Given a motivic space $\sF \in \MotSpc(S)$, consider the tautologically commuting square
  \[ \begin{tikzcd}
    j_\sharp^\MotSpc j^*_\MotSpc (\sF) \arrow{r}{\mrm{counit}}\arrow{d}{\mrm{unit}}
      & \sF \arrow{d}{\mrm{unit}} \\
    j_\sharp^\MotSpc j^*_\MotSpc i_*^\MotSpc i^*_\MotSpc (\sF) \arrow{r}{\mrm{counit}}
      & i_*^\MotSpc i^*_\MotSpc (\sF).
  \end{tikzcd} \]
Up to the canonical identification $j^*_\MotSpc i_*^\MotSpc \simeq \pt_U$ (\corref{cor:j^*i_* trivial}), this square is identified with a canonical commutative square
  \begin{equation} \label{eq:localization square}
    \begin{tikzcd}
      j_\sharp^\MotSpc j^*_\MotSpc (\sF) \arrow{r}\arrow{d} & \sF \arrow{d} \\
      j_\sharp^\MotSpc (\pt_U) \arrow{r} & i_*^\MotSpc i^*_\MotSpc (\sF)
    \end{tikzcd}
  \end{equation}
that we call the \emph{localization square} associated to $i$.
\end{constr}

\begin{thm}[Localization] \label{thm:localization}
Let $i : Z \hook S$ be a closed immersion of spectral algebraic spaces with quasi-compact open complement $j : U \hook S$.
Then for every motivic space $\sF \in \MotSpc(S)$, the localization square \eqref{eq:localization square} is cocartesian.
\end{thm}

The proof of \thmref{thm:localization} will be carried out in \secref{sec:loc}.
Here we record a few of its consequences.

\begin{cor} \label{cor:localization pointed}
Let $i : Z \hook S$ be a closed immersion of spectral algebraic spaces with quasi-compact open complement $j : U \hook S$.
For any pointed motivic space $\sF \in \MotSpc(S)_\bullet$, there is a cofibre sequence
  \begin{equation}
    j_\sharp j^*(\sF) \to \sF \to i_*i^*(\sF)
  \end{equation}
and a fibre sequence
  \begin{equation}
    i_*i^!(\sF) \to \sF \to j_*j^*(\sF)
  \end{equation}
in $\MotSpc(S)_\bullet$.
\end{cor}

\begin{proof}
The claim is that the commutative square
  \begin{equation*}
    \begin{tikzcd}
      j_\sharp^{\MotSpc_\bullet} j^*_{\MotSpc_\bullet} (\sF) \arrow{r}\arrow{d}
        & \sF \arrow{d}
      \\
      \pt_S \arrow{r}
        & i_*^{\MotSpc_\bullet} i^*_{\MotSpc_\bullet} (\sF)
    \end{tikzcd}
  \end{equation*}
is cocartesian in $\MotSpc(S)_\bullet$.
Since the forgetful functor $\MotSpc(S)_\bullet \to \MotSpc(S)$ reflects contractible colimits, it will suffice to show that the induced square of underlying motivic spaces
  \begin{equation*}
    \begin{tikzcd}
      j_\sharp^{\MotSpc} j^*_{\MotSpc} (\sF) \fibcoprod_{j_\sharp^{\MotSpc} j^*_{\MotSpc}(\pt_S)} \pt_S \arrow{r}\arrow{d}
        & \sF \arrow{d}
      \\
      \pt_S \arrow{r}
        & i_*^{\MotSpc}i^*_{\MotSpc} (\sF).
    \end{tikzcd}
  \end{equation*}
is cocartesian.
By \thmref{thm:localization}, the composite square
  \begin{equation*}
    \begin{tikzcd}
      j_\sharp^{\MotSpc} j^*_{\MotSpc} \sF \arrow{r}\arrow{d}
        & (j_\sharp^{\MotSpc} j^*_{\MotSpc} \sF) \fibcoprod_{j_\sharp^{\MotSpc} j^*_{\MotSpc}(\pt_S)} \pt_S \arrow{r}\arrow{d}
        & \sF \arrow{d}
      \\
      j_\sharp^{\MotSpc} j^*_{\MotSpc}(\pt_S) \arrow{r}
        & \pt_S \arrow{r}
        & i_*^{\MotSpc}i^*_{\MotSpc} \sF.
    \end{tikzcd}
  \end{equation*}
is cocartesian.
Since the left-hand square is evidently cocartesian, it follows that the right-hand square is also cocartesian.
\end{proof}

The following reformulation of the localization theorem is an analogue of Kashiwara's lemma in the setting of D-modules.

\begin{thm} \label{thm:kashiwara}
Let $i : Z \hook S$ be a closed immersion of spectral algebraic spaces with quasi-compact open complement.
Then the direct image functor $i_*^{\MotSpc}$ is fully faithful, and its essential image is spanned by objects $\sF\in\MotSpc(S)$ whose restriction $j^*_{\MotSpc}(\sF) \in \MotSpc(U)$ is contractible.
\end{thm}

\begin{proof}
First we show that $i_*$ is fully faithful.
An application of \thmref{thm:localization} to the motivic space $i_*(\sF) \in \MotSpc(S)$ shows that the co-unit induces an isomorphism $i_*i^*i_* \to i_*$.
By a standard argument it therefore suffices to show that $i_*$ is conservative.
For this let $\varphi : \sF_1 \to \sF_2$ be a morphism in $\MotSpc(Z)$ such that $i_*(\varphi)$ is invertible.
To show that $\varphi$ is invertible, it will suffice to show that
  \begin{equation*}
    \Gamma(X, \sF_1) \to \Gamma(X, \sF_2)
  \end{equation*}
is invertible for each $X \in \Sm_{/Z}$.
By \propref{prop:Nisnevich lifting}, we may assume that $X$ is the base change of some $Y \in\Sm_{/S}$.
Then we have natural isomorphisms $\Gamma(X, \sF_k) \simeq \Gamma(Y, i_*(\sF_k))$ for each $k$, so the claim follows.

Next we identify the essential image of $i_*$.
Suppose that $\sF \in \MotSpc(S)$ with $j^*(\sF)$ contractible.
Then \thmref{thm:localization} yields that the unit map $\sF \to i_*i^*(\sF)$ is invertible, so that $\sF$ belongs to the essential image of $i_*$.
The other inclusion follows from \corref{cor:j^*i_* trivial}.
\end{proof}

\begin{rem}\label{rem:i^! concrete}
We can use the localization theorem to give a concrete description of the abstractly defined functor $i^!_{\MotSpc_\bullet}$.
Namely, it is given by
  \begin{equation*}
    i^!_{\MotSpc_\bullet}(\sF) \simeq \Fib(i^*(\sF) \to i^*j_*j^*(\sF))
  \end{equation*}
for any $\sF \in \MotSpc(Z)_\bullet$.
\end{rem}

\begin{cor}[Nilpotent invariance] \label{cor:nilpotent invariance}
Let $i : S_0 \hook S$ be a closed immersion of spectral algebraic spaces.
Suppose that $i$ induces an isomorphism $(S_0)_{\cl,\red} \simeq S_{\cl,\red}$ of reduced classical algebraic spaces.
Then the functors $i^*_{\MotSpc}$ and $i_*^{\MotSpc}$ are mutually inverse equivalences.
\end{cor}

\begin{proof}
Since the complement of $i$ is empty, this follows from \thmref{thm:kashiwara}.
\end{proof}

\begin{cor} \label{cor:nilpotent invariance S_cl}
Let $S$ be a spectral algebraic space, and let $i : S_\cl \hook S$ denote the inclusion of its underlying classical algebraic space (viewed as a discrete spectral algebraic space).
Then the functors $i^*_{\MotSpc}$ and $i_*^{\MotSpc}$ define mutually inverse equivalences $\MotSpc(S) \simeq \MotSpc(S_\cl)$.
\end{cor}

\begin{warn}
\corref{cor:nilpotent invariance S_cl} does not assert that $\MotSpc(S)$ is equivalent to the classical motivic homotopy category over $S_\cl$, denoted $\MotSpc_\cl(S)$ in \propref{prop:char 0}.
\end{warn}

\begin{cor}\label{cor:nilpotent invariance affine}
Let $R$ be a \cEring and consider the affine spectral scheme $S = \Spec(R)$.
Denote by $i : S_\cl \to S$ the inclusion of the underlying classical scheme.
Then the functor $i_*^{\MotSpc}$ induces an equivalence from the \inftyCat of $\A^1$-homotopy invariant Nisnevich sheaves on $\Sm^\aff_{/S_\cl}$ to the \inftyCat of $\A^1$-homotopy invariant Nisnevich sheaves of spaces on $\Sm^\aff_{/S}$.
\end{cor}

\begin{proof}
Combine \corref{cor:nilpotent invariance S_cl} with \corref{cor:MotSpc on Sm^aff}.
\end{proof}

\ssec{Closed base change formula}
\label{ssec:closed/bc}

\begin{constr}
Consider a cartesian square
  \begin{equation*}
    \begin{tikzcd}
      Y \arrow[hookrightarrow]{r}{k}\arrow{d}{g}
        & X \arrow{d}{f} \\
      Z \arrow[hookrightarrow]{r}{i}
        & S,
    \end{tikzcd}
  \end{equation*}
of spectral algebraic spaces, where $i$ and $k$ are closed immersions with quasi-compact open complements.
Then the square
  \begin{equation*}
    \begin{tikzcd}
      \MotSpc(Z)\ar{r}{i_*}\ar{d}{g^*}
        & \MotSpc(S)\ar{d}{f^*}
      \\
      \MotSpc(Y)\ar{r}{k_*}
        & \MotSpc(X)
    \end{tikzcd}
  \end{equation*}
commutes up to a natural transformation
  \begin{equation*}
    f^* i_* \xrightarrow{\mrm{unit}}
      k_* k^* f^* i_* \simeq
      k_* g^* i^* i_* \xrightarrow{\mrm{counit}}
      k_* g^*.
  \end{equation*}
\end{constr}

\begin{prop} \label{prop:closed base change}
Suppose given a cartesian square of spectral algebraic spaces as above, with $i$ and $k$ closed immersions with quasi-compact open complements.
Then the canonical natural transformation
  \begin{equation*}
    f^*_{\MotSpc} i_*^{\MotSpc} \to k_*^{\MotSpc}g^*_{\MotSpc}
  \end{equation*}
is invertible.
\end{prop}

\begin{proof}
Since $i_*$ is fully faithful (\thmref{thm:kashiwara}) it suffices to show that the natural transformation
  \begin{equation*}
    f^* i_*i^* \to k_*g^*i^*
  \end{equation*}
is invertible.
This follows by a straightforward application of the localization theorem (\thmref{thm:localization}) for the closed immersions $i$ and $k$, respectively, using the smooth base change formula (\propref{prop:smooth base change}) for the open complements.
\end{proof}

Recall that in the pointed setting, the functor $i_*^{\MotSpc_\bullet} : \MotSpc(Z)_\bullet \to \MotSpc(S)_\bullet$ admits a right adjoint $i^!_{\MotSpc_\bullet}$ (\corref{cor:i_* admits right adjoint i^!}).

\begin{cor} \label{cor:closed base change pointed}
Suppose given a cartesian square of spectral algebraic spaces as above, with $i$ and $k$ closed immersions with quasi-compact open complements.
Then the natural transformations
  \begin{align*}
    k_*^{\MotSpc_\bullet}g^*_{\MotSpc_\bullet} \to f^*_{\MotSpc_\bullet} i_*^{\MotSpc_\bullet}
    \\
    i^!_{\MotSpc_\bullet} f_*^{\MotSpc_\bullet} \to g_*^{\MotSpc_\bullet} k^!_{\MotSpc_\bullet}
  \end{align*}
are invertible.
\end{cor}

\begin{proof}
The first follows immediately from the unpointed version (\propref{prop:closed base change}), and the second follows from the first by passing to right adjoints.
\end{proof}

\ssec{Closed projection formula}
\label{ssec:closed/proj}

\begin{constr}
Let $i : Z \hook S$ be a closed immersion with quasi-compact open complement.
Given pointed motivic spaces $\sF' \in \MotSpc(Z)_\bullet$ and $\sF \in \MotSpc(S)_\bullet$, we get a morphism
  \begin{equation*}
    i^*_{\MotSpc_\bullet}(i_*^{\MotSpc_\bullet}(\sF') \wedge \sF)
      \simeq i^*_{\MotSpc_\bullet}i_*^{\MotSpc_\bullet}(\sF') \wedge i^*_{\MotSpc_\bullet}(\sF)
      \xrightarrow{\mrm{counit}} \sF' \wedge i^*_{\MotSpc_\bullet}(\sF)
  \end{equation*}
which corresponds by adjunction to a canonical morphism
  \begin{equation*}
    i_*^{\MotSpc_\bullet}(\sF') \wedge \sF \to i_*^{\MotSpc_\bullet}(\sF' \wedge i^*_{\MotSpc_\bullet}(\sF)).
  \end{equation*}
\end{constr}

\begin{prop} \label{prop:closed projection formula}
Let $i : Z \hook S$ be a closed immersion with quasi-compact open complement.
Given pointed motivic spaces $\sF' \in \MotSpc(Z)_\bullet$, $\sF \in \MotSpc(S)_\bullet$, and $\sG \in \MotSpc(S)_\bullet$, there are canonical bifunctorial isomorphisms
  \begin{align*}
    i_*^{\MotSpc_\bullet}(\sF') \wedge \sF \to i_*^{\MotSpc_\bullet}(\sF' \wedge i^*_{\MotSpc_\bullet}(\sF)),\\
    i^!_{\MotSpc_\bullet} \uHom_S(\sG, \sF) \to \uHom_Z(i^*_{\MotSpc_\bullet}(\sG), i^!_{\MotSpc_\bullet}(\sF)).
  \end{align*}
\end{prop}

\begin{proof}
The second statement follows from the first by passing to right adjoints.
For the first, it suffices by fully faithfulness of $i_*$ (\thmref{thm:kashiwara}) to show that the morphism
  \begin{equation*}
    i_*(i^*\sF) \wedge \sG \to i_*(i^*\sF \wedge i^*\sG)
  \end{equation*}
is invertible for all pointed motivic spaces $\sF, \sG \in \MotSpc(S)_\bullet$.
This follows from the localization theorem (\corref{cor:localization pointed}), using the smooth projection formula (\propref{prop:smooth projection formula}) for the open complement $j : U \hook S$.
\end{proof}

\ssec{Smooth-closed base change formula}
\label{ssec:closed/smbc}

\begin{constr}
Consider a cartesian square of spectral algebraic spaces
  \begin{equation*}
    \begin{tikzcd}
      Y \arrow[hookrightarrow]{r}{k}\arrow{d}{q}
        & X \arrow{d}{p} \\
      Z \arrow[hookrightarrow]{r}{i}
        & S,
    \end{tikzcd}
  \end{equation*}
where $i$ and $k$ are closed immersions with quasi-compact open complements, and $p$ and $q$ are smooth.
Then by \propref{prop:smooth base change} it follows that the square
  \begin{equation*}
    \begin{tikzcd}
      \MotSpc(Y)_\bullet \ar{r}{k_*}\ar{d}{q_\sharp}
        & \MotSpc(X)_\bullet \ar{d}{p_\sharp}
      \\
      \MotSpc(Z)_\bullet \ar{r}{i_*}
        & \MotSpc(S)_\bullet.
    \end{tikzcd}
  \end{equation*}
commutes up to a natural transformation
  \begin{equation*}
    p_\sharp k_*
      \xrightarrow{\mrm{unit}} i_* i^* p_\sharp k_*
      \simeq i_* q_\sharp k^* k_*
      \xrightarrow{\mrm{counit}} i_* q_\sharp
  \end{equation*}
\end{constr}

\begin{prop} \label{prop:(smooth,closed)-base change}
Suppose given a cartesian square of spectral algebraic spaces as above, where $i$ and $k$ are closed immersions with quasi-compact open complements, and $p$ and $q$ are smooth.
Then the canonical natural transformations
  \begin{align*}
    p_\sharp^{\MotSpc_\bullet} k_*^{\MotSpc_\bullet} \to i_*^{\MotSpc_\bullet} q_\sharp^{\MotSpc_\bullet}\\
    q^*_{\MotSpc_\bullet} i^!_{\MotSpc_\bullet} \to k^!_{\MotSpc_\bullet} p^*_{\MotSpc_\bullet}
  \end{align*}
are invertible.
\end{prop}

\begin{proof}
The second transformation is obtained from the first by passing to right adjoints.
Since the direct image functor $k_*$ is fully faithful (\thmref{thm:kashiwara}), it suffices to show that the transformation $p_\sharp k_* k^* \to i_* q_\sharp k^*$, obtained by pre-composition with $k^*$, is invertible.
This follows directly from the localization theorem (\corref{cor:localization pointed}) and the smooth base change formula (\propref{prop:smooth base change}).
\end{proof}


\section{Proof of the localization theorem}
\label{sec:loc}

This section is dedicated to the proof of our main result, \thmref{thm:localization}.
For the duration of the section, we fix a closed immersion of spectral algebraic spaces $i : Z \hook S$ with quasi-compact open complement $j : U \hook S$.
Given $X \in \Sm_{/S}$, we will use the notation $X_U := X \fibprod_S U \in \Sm_{/U}$ and $X_Z := X \fibprod_S Z \in \Sm_{/Z}$.

\ssec{The space of \texorpdfstring{$Z$}{Z}-trivialized maps}
\label{ssec:loc/thespace}

In this paragraph we formulate \propref{prop:h(X,t) is contractible}, which aside from \thmref{thm:i_* commutes with contractible colimits} is the main input that goes into the proof of the localization theorem; \ssecref{ssec:loc/contract} will be dedicated to its proof.

\begin{constr}\label{constr:h^Z_S(X)}
Let $X \in \Sm_{/S}$ and denote by $\h^Z_S(X) \in \Spc(S)$ the fibred space
  \begin{equation*}
    \h^Z_S(X) := \hspc[S]{X} \fibcoprod_{\hspc[S]{X_U}} \hspc[S]{U}.
  \end{equation*}
Note that for $Y \in \Sm_{/S}$, the space $\Gamma(Y, \hspc[S]{U})$ is either empty or contractible depending on whether $Y_Z \in \Sm_{/Z}$ is empty.
It follows that the space of sections $\Gamma(Y, \h_S^Z(X))$ is contractible when $Y_Z$ is empty, and otherwise is given by the mapping space $\Map_S(Y,X)$.
\end{constr}

\begin{rem}
There is a canonical identification $i^*_{\Spc}(\h_S^Z(X)) \simeq \hspc[Z]{X_Z}$ (since $i^*_{\Spc}$ commutes with colimits).
\end{rem}

\begin{constr}
Let $X\in\Sm_{/S}$ and $t : Z \hook X$ an $S$-morphism, i.e. a partially defined section of $X \to S$.
Then $t$ corresponds by adjunction to a canonical morphism $\tau : \pt_S = \hspc[S]{S} \to i_*^{\Spc}(\hspc[S]{X_Z})$, and we define the fibred space $\hspc[S]{X, t} \in \Spc(S)$ as the fibre of the unit map
  \begin{equation} \label{eq:counit on hps^Z(X)}
    \h_S^Z(X) \to i_*^{\Spc}i^*_{\Spc}(\h_S^Z(X)) \simeq i_*^{\Spc}(\hspc[Z]{X_Z})
  \end{equation}
at the point $\tau$.
Thus for any $Y \in \Sm_{/S}$, the space $\Gamma(Y, \hspc[S]{X,t})$ is contractible when $Y_Z$ is empty, and otherwise is given by the fibre of the restriction map
  \[ \Map_S(Y, X) \to \Map_Z(Y_Z, X_Z) \]
at the point defined by the composite $Y_Z \to Z \stackrel{t}{\hook} X_Z$.
\end{constr}

\begin{rem}
Informally speaking, points of the space $\Gamma(Y, \hspc[S]{X, t})$ (when $Y_Z \ne \initial$) are pairs $(f, \alpha)$, with $f : Y \to X$ an $S$-morphism and $\alpha$ a commutative triangle
  \begin{equation*}
    \begin{tikzcd}
      Y_Z \arrow{r}{f_Z}\arrow{d}
        & X_Z.
      \\
      Z \arrow[hookrightarrow]{ru}{t}
    \end{tikzcd}
  \end{equation*}
We refer to $\alpha$ informally as a \emph{$Z$-trivialization} of $f$.
\end{rem}

\begin{rem}\label{rem:h(X,t) smooth base change}
For a smooth morphism $p : T \to S$, there is a canonical isomorphism
  \begin{equation*}
    p^*_{\Spc}(\hspc[S]{X, t}) \simeq \hspc[T]{X_T, t_T},
  \end{equation*}
where $t_T : Z_T \hook X_T$ is the base change of $t$ along $p$.
This follows from the fact that the functor $p^*_{\Spc}$ commutes with limits and colimits.
\end{rem}

We will deduce \thmref{thm:localization} from the following result:

\begin{prop} \label{prop:h(X,t) is contractible}
Let $Z \hook S$ be a closed immersion of spectral algebraic spaces with quasi-compact open complement.
Let $p : X \to S$ be a smooth morphism of spectral algebraic spaces, and $t : Z \hook X$ an $S$-morphism.
If $X$ is affine and admits an étale morphism to $\A^n_S$, for some $n\ge 0$, then the space $\hspc[S]{X,t}$ is motivically contractible.
That is, the morphism $\hspc[S]{X,t} \to \pt_S$ is a motivic equivalence.
\end{prop}

The proof will be completed in \ssecref{ssec:loc/contract}.

\ssec{Motivic contractibility of \texorpdfstring{$\hspc[S]{X,t}$}{h(X,t)}}
\label{ssec:loc/contract}

In this paragraph we prove \propref{prop:h(X,t) is contractible}.
We will need the local structure theory of quasi-smooth closed immersions.

\begin{prop}\label{prop:quasi-smooth}
Let $k : Y \hook X$ be a closed immersion of spectral algebraic spaces.
Let $\sL_{Y/X}$ denote the relative cotangent complex.
Then the following conditions are equivalent:

\noindent{\em(i)}
The morphism $k$ is locally of finite presentation, and the quasi-coherent $\sO_Y$-module $\sL_{Y/X}[-1]$ is locally free of finite rank.

\noindent{\em(ii)}
The morphism $k$ is almost of finite presentation, and the quasi-coherent $\sO_Y$-module $\sL_{Y/X}[-1]$ is locally free of finite rank.

\noindent{\em(iii)}
The morphism of underlying classical schemes $k_\cl : Y_\cl \hook X_\cl$ is of finite presentation, and the quasi-coherent $\sO_Y$-module $\sL_{Y/X}[-1]$ is locally free of finite rank.

\noindent{\em(iv)}
Nisnevich-locally on $X$, there exists a morphism $f : X \to \A^n_{\bS}$ fitting in a cartesian square
  \begin{equation*}
    \begin{tikzcd}
      Y \ar{r}{k}\ar{d}
        & X \ar{d}{f}
      \\
      \Spec(\bS) \ar{r}{\{0\}}
        & \A^n_{\bS},
    \end{tikzcd}
  \end{equation*}
where the lower horizontal arrow is the inclusion of the origin in spectral affine space.
\end{prop}

\begin{defn}
If $k : Y \hook X$ satisfies one of the equivalent conditions of \propref{prop:quasi-smooth}, then we say that it is \emph{quasi-smooth}, and write $\sN_{Y/X} := \sL_{Y/X}[-1]$ for its \emph{conormal sheaf}.
\end{defn}

\begin{proof}[Proof of \propref{prop:quasi-smooth}]
The claim being local, we may assume that $X$ is affine and that $\sL_{Y/X}[-1]$ is free of rank $n\ge 0$.
The equivalence between the first three conditions follows from \cite[Thm.~7.4.3.18]{HA-20170918}.
The implication (iv) $\implies$ (i) is obvious.
Suppose that (i) holds and choose a basis $df_1,\ldots,df_n$ for $\Gamma(Y, \sL_{Y/X}[-1])$.
If $\sI$ denotes the fibre of the morphism $\sO_{X} \to k_*\sO_{Y}$, then there is a canonical isomorphism $\pi_0(k^*\sI) \simeq \pi_1(\sL_{Y/X})$ of $\pi_0(\sO_{Y})$-modules \cite[Thm.~7.4.3.1]{HA-20170918}.
The global sections $df_i$ correspond to global sections $\bar{f}_i$ of $k^*(\sI)$, which we can lift along the surjection $\sI \to k_*k^*(\sI)$ to global sections $f_i$ of $\sI$.
By Nakayama's lemma we may assume that these $f_i$ generate $\pi_0(\sI)$ as a $\pi_0(\sO_{X})$-module.
Therefore, they determine a morphism $f : X \to \A^n_{\bS}$ and a commutative square
  \begin{equation*}
    \begin{tikzcd}
      Y \ar{r}{k}\ar{d}
        & X \ar{d}{f}
      \\
      \Spec(\bS)\ar{r}{\{0\}}
        & \A^n_\bS
    \end{tikzcd}
  \end{equation*}
which is cartesian on underlying classical schemes.
To show that it is cartesian itself, it will suffice by \cite[Cor.~7.4.3.4]{HA-20170918} to show that the relative cotangent complex $\sL_{Y/V}$ of the morphism $Y \to V := X \fibprod_{\A^n_{\bS}} \Spec(\bS)$  vanishes.
But this follows immediately from an inspection of the exact triangle
  \begin{equation*}
    \sL_{V/X}|Y \to \sL_{Y/X} \to \sL_{Y/V}
  \end{equation*}
where both first terms are isomorphic to a shifted free module $\sO_Y^{\oplus n}[1]$.
\end{proof}

We will actually use a slight variant of \propref{prop:quasi-smooth} which concerns the case of a smooth morphism admitting a globally defined section (such a section is automatically quasi-smooth).

\begin{lem} \label{lem:linear approximation}
Let $p : X \to S$ be a smooth morphism of spectral algebraic spaces, and suppose it admits a section $s : S \hook X$.
Then Nisnevich-locally on $X$, there exists an $S$-morphism $f : X \to \A^n_S$ fitting in a cartesian square
  \[ \begin{tikzcd}
    S \arrow[hookrightarrow]{r}{s}\arrow[equals]{d}
      & X \arrow{d}{f} \\
    S \arrow[hookrightarrow]{r}{\{0\}}
      & \A^n_S.
  \end{tikzcd} \]
Moreover, the morphism $f$ is \'etale on some Zariski neighbourhood of $s$.
\end{lem}

\begin{proof}
Note that there is a canonical isomorphism $\sL_{S/X}[-1] \simeq \sL_{X/S}|S$.
By the assumption, $\sL_{X/S}$ is free of rank $n$, so the same holds for the $\sO_S$-module $\sN_{S/X} = \sL_{S/X}[-1]$.
Then the first claim is proven exactly as the implication (i) $\implies$ (iv) of \propref{prop:quasi-smooth}.
For the second, note that the canonical isomorphism $s^*\sL_{X/\A^n_S} \simeq \sL_{S/S} \simeq 0$ shows that $f$ is étale on the image of $s$ (since $f$ is of finite presentation).
In other words, $s$ factors through the étale locus of $f$.
\end{proof}

Next we apply the structure theory to lift partially defined sections (in a weak sense).

\begin{lem} \label{lem:lifting sections}
Let $i : Z \hook S$ be a closed immersion of spectral algebraic spaces, $p : X \to S$ a smooth morphism of spectral algebraic spaces, and $t : Z \hook X$ an $S$-morphism.
Then Nisnevich-locally on $X$, there exists a spectral algebraic space $Y$ over $X$ such that the composite $Y \to S$ is étale, and induces an isomorphism $Y_Z \to Z$:
  \begin{equation*}
    \begin{tikzcd}
      Z \ar{r}\ar{d}{t}
        & Y \ar{d}
      \\
      X_Z \ar{r}{i'}\ar{d}
        & X \ar{d}{p}
      \\
      Z \ar{r}{i}
        & S
    \end{tikzcd}
  \end{equation*}
\end{lem}

\begin{proof}
Applying \lemref{lem:linear approximation} to the smooth morphism $X_Z \to Z$ with section $Z \hook X_Z$ induced by $t$, we obtain a cartesian square
  \begin{equation*}
    \begin{tikzcd}
      Z \ar{r}{t}\ar[equals]{d}
        & X_Z \ar{d}{f}
      \\
      Z\ar{r}{\{0\}}
        & \A^n_Z.
    \end{tikzcd}
  \end{equation*}
The morphism $f$ is determined by a set of global sections $f_1,\ldots,f_n$ of $\sO_{X_Z}$; lifting them along the surjection $\sO_X \to i'_*(\sO_{X_Z})$, we obtain global sections $g_i$ of $\sO_X$ which determine a morphism $g : X \to \A^n_S$.
We define $Y_0 := X \fibprod_{\A^n_S} S$ so that we have a commutative square
  \begin{equation*}
    \begin{tikzcd}
      Z \ar{r}{i''}\ar{d}
        & Y_0 \ar{d}
      \\
      X_Z \ar{r}{i'}
        & X.
    \end{tikzcd}
  \end{equation*}
where the morphism $Z \to (Y_0)_Z$ is an isomorphism on underlying classical schemes.
The exact triangle
  \begin{equation*}
    \sL_{(Y_0)_Z/Y_0}|Z \to \sL_{Z/Y_0} \to \sL_{Z/(Y_0)_Z}
  \end{equation*}
shows that $\sL_{Z/(Y_0)_Z}$ vanishes, so it follows from \cite[Cor.~7.4.3.4]{HA-20170918} that the square is cartesian.
Then we have a canonical isomorphism $(i'')^*\sL_{Y_0/S} \simeq \sL_{Z/Z} \simeq 0$, which shows that $Y_0 \to S$ is étale on the image of $Z$.
Thus we may take $Y \hook Y_0$ to be the étale locus of $Y_0 \to S$ to conclude.
\end{proof}

We are now ready to return to our study of the space of $Z$-trivialized maps.
We first deal with the case of vector bundles:

\begin{lem} \label{lem:h(E,s) of vector bundle}
Let $E$ be a vector bundle over $S$ with zero section $s : S \hook E$.
Then the space $\hspc[S]{E, s_Z}$ is motivically contractible, where $s_Z : Z \hook E_Z$ denotes the base change of $s$ along $i : Z \hook S$.
\end{lem}

\begin{proof}
The map $\varphi : \hspc[S]{E, s_Z} \to \pt_S$ admits a section $\sigma : \pt_S \to \hspc[S]{E, s_Z}$, induced by the composite $\hspc[S]{S} \xrightarrow{s} \hspc[S]{E} \to \h_S^Z(E)$.
It will suffice to exhibit an $\A^1$-homotopy
  \[ \gamma : \hspc[S]{\A^1_S} \times \hspc[S]{E, s_Z} \to \hspc[S]{E, s_Z}\]
between $\sigma\circ\varphi$ and the identity.
The canonical action of $\A^1_S$ on $E$ gives rise to the vertical maps in the commutative square
  \begin{equation*}
    \begin{tikzcd}
      \hspc[S]{\A^1_S} \times \h_S^Z(E) \ar{r}\ar{d}
        & \hspc[S]{\A^1_S} \times i_*^\Spc(\hspc[Z]{E_Z})\ar{d}
      \\
      \h_S^Z(E)\ar{r}
        & i_*^\Spc(\hspc[Z]{E_Z}).
    \end{tikzcd}
  \end{equation*}
The homotopy $\gamma$ is the map induced on fibres (given informally by the assignment $(\lambda : Y \to \A^1_S, f : Y \to E) \mapsto (\lambda\cdot f : Y \to E)$ on $Y$-sections).
\end{proof}

Our final ingredient is a certain invariance property for the construction $\hspc[S]{X,t}$:

\begin{lem} \label{lem:etale invariance of h(X,t)}
Let $X$, resp. $X'$, be a smooth spectral algebraic space over $S$, and let $t : Z \hook X$, resp. $t' : Z \hook X'$, be an $S$-morphism.
Suppose $f : X' \to X$ is an étale $S$-morphism such that the square
  \begin{equation*}
    \begin{tikzcd}
      Z \arrow[hookrightarrow]{r}{t'}\arrow[equals]{d}
        & X'_Z \arrow{d}{f_Z}
      \\
      Z \arrow[hookrightarrow]{r}{t}
        & X_Z
    \end{tikzcd}
  \end{equation*}
is cartesian.
Then the induced morphism of fibred spaces
  \begin{equation*}
    \varphi : \hspc[S]{X', t'} \to \hspc[S]{X, t}.
  \end{equation*}
is a Nisnevich-local equivalence.
\end{lem}

\begin{proof}
The claim is that the induced morphism of Nisnevich sheaves $\LNis(\varphi)$ is invertible, so it will suffice to show that it is $0$-truncated (i.e. its diagonal is a monomorphism) and $0$-connected (i.e. it is an effective epimorphism and so is its diagonal).

\noindent\emph{Step 1}.
To show that $\LNis(\varphi)$ is 0-truncated, it suffices to show that $\varphi$ is already $0$-truncated (since $\LNis$ is exact).
For this, it suffices to show that for every $Y\in\Sm_{/S}$, the induced morphism of spaces of $Y$-sections
  \[ \Gamma(Y, \varphi) : \Gamma(Y, \h^Z_S(X',t')) \to \Gamma(Y, \h^Z_S(X,t)) \]
is 0-truncated.
We may assume $Y_Z$ is not empty; then this is the morphism induced on fibres in the diagram
  \begin{equation*}
    \begin{tikzcd}
      \Gamma(Y, \h^Z_S(X',t')) \arrow{r}\arrow[dashed]{d}
        & \Map_S(Y, X') \arrow{r}\arrow{d}
        & \Map_Z(Y_Z, X'_Z) \arrow{d} \\
      \Gamma(Y, \h^Z_S(X,t)) \arrow{r}
        & \Map_S(Y, X) \arrow{r}
        & \Map_Z(Y_Z, X_Z)
    \end{tikzcd}
  \end{equation*}
Note that the two right-hand vertical morphisms are 0-truncated: $p$ is itself 0-truncated since it is \'etale, and since the Yoneda embedding commutes with limits, the induced morphism $\h_S(X') \to \h_S(X)$ is also 0-truncated.
It follows that the left-hand vertical morphism is also 0-truncated for each $Y$.

\noindent\emph{Step 2}.
To show that $\LNis(\varphi)$ is an effective epimorphism, it suffices to show that for every $Y\in\Sm_{/S}$ (with $Y_Z$ not empty), any $Y$-section of $\h^Z_S(X,t)$ can be lifted Nisnevich-locally along $\varphi$.
Let $f$ be a $Y$-section of $\h^Z_S(X,t)$, i.e. a $Z$-trivialized morphism $f : Y \to X$.
Let $q : Y' \to Y$ denote the base change of $p : X' \to X$ along $f$:
  \begin{equation*}
    \begin{tikzcd}
      Y' \arrow{r}{q}\arrow{d}{g}
        & Y \arrow{d}{f} \\
      X' \arrow{r}{p}
        & X.
    \end{tikzcd}
  \end{equation*}
Then note that
  \begin{equation*}
    \begin{tikzcd}
      q^{-1}(Y_U) \arrow[hookrightarrow]{r}\arrow{d}
        & Y' \arrow{d}{q} \\
      Y_U \arrow[hookrightarrow]{r}
        & Y
    \end{tikzcd}
  \end{equation*}
is a Nisnevich square.
Indeed, the closed immersion $Y_Z \hook Y$ is complementary to $Y_U \hook Y$, and it is clear that $q^{-1}(Y_Z) \to Y_Z$ is invertible because in the diagram
  \begin{equation*}
    \begin{tikzcd}
      q^{-1}(Y_Z) \arrow{r}\arrow{d}
        & Y_Z \arrow{d} \\
      Z \arrow{r}{\id_Z}\arrow[hookrightarrow]{d}{t}
        & Z \arrow[hookrightarrow]{d}{t'} \\
      X'_Z \arrow{r}{p_Z}
        & X_Z
    \end{tikzcd}
  \end{equation*}
the lower square and the composite square are cartesian, and hence so is the upper square.
It now suffices to show that the restriction of $f$ to either component of this Nisnevich cover lifts to $\h^Z_S(X',t')$.
The restriction $f|Y'$ lifts to a section of $\h^Z_S(X',t')$ given by $g : Y' \to X'$.
The restriction $f|Y_U$ admits a lift tautologically: since $(Y_U) \fibprod_S Z = \initial$, the spaces $\h^Z_S(X,t)(Y_U)$ and $\h^Z_S(X',t')(Y_U)$ are both contractible.

\noindent\emph{Step 3}.
It remains to show that the diagonal $\Delta_{\LNis(\varphi)}$ of $\LNis(\varphi)$ is an effective epimorphism, or equivalently that $\LNis(\Delta_\varphi)$ is.
For every $Y\in\Sm_{/S}$, the diagonal induces a morphism of spaces
  \[ \Gamma(Y, \h^Z_S(X',t')) \to \Gamma(Y, \h^Z_S(X',t')) \fibprod_{\Gamma(Y, \h^Z_S(X,t))} \Gamma(Y, \h^Z_S(X',t')). \]
It suffices to show that for each $Y$ (with $Y_Z$ not empty), any point of the target lifts Nisnevich-locally to a point of the source.
Choose a point of the target, given by two $Z$-trivialized morphisms $f : Y \to X'$ and $g : Y \to X'$, and an identification $\alpha : p \circ f \simeq p \circ g$.
Let $Y_0 \hook Y$ denote the open immersion defined as the equalizer of the pair $(f,g)$; note that the closed immersion $Y_Z \hook Y$ factors through $Y_0$.
Thus the open immersions $Y_0 \hook Y$ and $Y_U \hook Y$ form a Zariski cover of $Y$.
It is clear that the point $(f,g,\alpha)$ lifts after restriction to $Y_0$ by definition, and after restriction to $Y_U$ since $Y_U \fibprod_S Z = \initial$, so the claim follows.
\end{proof}

We are now ready to conclude the proof of \propref{prop:h(X,t) is contractible}.

\begin{proof}[Proof of \propref{prop:h(X,t) is contractible}]
The question being local on $X$, we may assume by \lemref{lem:lifting sections} that there exists a Nisnevich square
  \begin{equation} \label{eq:post-reduction-nis-square}
    \begin{tikzcd}
      Y_U \arrow[hookrightarrow]{r}\arrow{d}
        & Y \arrow{d}{q}
      \\
      U \arrow[hookrightarrow]{r}{j}
        & S
    \end{tikzcd}
  \end{equation}
where $j : U \hook S$ is the complement of $i$, and $q$ factors through $p : X \to S$.
By the Nisnevich separation property (\propref{prop:Nisnevich separation of MotSpc}), it will suffice to show that $j^*\hspc[S]{X,t}$ and $q^* \hspc[S]{X,t}$ are motivically contractible.
By \remref{rem:h(X,t) smooth base change}, we have $j^*\hspc[S]{X,t} \simeq \hspc[U]{X_U,t_U}$.
But since $U$ is complementary to $Z$, $t_U$ is the inclusion of the empty scheme, so $\hspc[U]{X_U,t_U}$ is contractible by construction.
Next consider $q^* \hspc[S]{X,t} \simeq \hspc[Y]{Y \fibprod_S X, t'}$, where $t' : Y_Z \hook (Y \fibprod_S X)_Z$ is the base change of $t$.
By construction there exists a section $t'': Y \hook Y \fibprod_S X$ which lifts $t'$ (since $q$ factors through $X$).
Hence by \lemref{lem:linear approximation} and \lemref{lem:etale invariance of h(X,t)}, we have a motivic equivalence
  \[ \hspc[Y]{Y \fibprod_S X, t'} \simeq \hspc[Y]{\A^n_Y, z}, \]
where $z : Y_Z \hook \A^n_{Y_Z}$ denotes the zero section.
Now the claim follows from \lemref{lem:h(E,s) of vector bundle}.
\end{proof}

\ssec{The proof}
\label{ssec:loc/proof}

We conclude this section by proving the localization theorem (\thmref{thm:localization}).
Let $i : Z \hook S$ be a closed immersion of spectral algebraic spaces with quasi-compact open complement $j : U \hook S$.
We wish to show that for every motivic space $\sF \in \MotSpc(S)$, the canonical morphism
  \begin{equation} \label{eq:localization 1}
    \sF \fibcoprod_{j_\sharp^\MotSpc j^*_\MotSpc(\sF)} \hmot[S]{U} \to i_*^\MotSpc i^*_\MotSpc(\sF)
  \end{equation}
is invertible.
In what follows below, we will omit the decorations $\MotSpc$.

Using \propref{prop:MotSpc generated by affines} we may reduce to the case where $\sF$ is the motivic localization $\hmot[S]{X}$ of some affine $X\in\Sm_{/S}$ that admits an étale morphism $X \to \A^n_S$ for some $n\ge 0$, since all functors involved commute with sifted colimits (\thmref{thm:i_* commutes with contractible colimits}).
In this case the morphism \eqref{eq:localization 1} is canonically identified with the morphism
  \begin{equation*}
    \hmot[S]{X} \fibcoprod_{\hmot[S]{X_U}} \hmot[S]{U} \to i_*^\MotSpc \hmot[Z]{X_Z}
  \end{equation*}
where we write $X_U = X \fibprod_S U$ and $X_Z = X \fibprod_S Z$.
Note that the source of this morphism is the motivic localization of the space $\h_S^Z(X)$ (\constrref{constr:h^Z_S(X)}).
Hence it suffices to show that the morphism of fibred spaces
  \begin{equation*}
    \h_S^Z(X) \to i_*^{\Spc} \hspc[Z]{X_Z}
  \end{equation*}
is a motivic equivalence.
By universality of colimits, it suffices to show that for every $Y\in\Sm_{/S}$ and every morphism $\hspc[S]{Y} \to i_*^{\Spc} \hspc[Z]{X_Z}$, corresponding to an $S$-morphism $t : Z \to X$, the base change
  \begin{equation*}
    \h_S^Z(X) \fibprod_{i_*^{\Spc} \hspc[Z]{X_Z}} \hspc[S]{Y} \to \hspc[S]{Y}
  \end{equation*}
is invertible.
If $p : Y \to S$ denotes the structural morphism, we have $\hspc[S]{Y} \simeq p_\sharp^{\Spc} \hspc[Y]{Y}$ so that this morphism is identified, by the smooth projection formula (\propref{prop:smooth projection formula}), with a morphism
  \begin{equation*}
    p_\sharp^{\Spc}(p^*_{\Spc}\h_S^Z(X) \fibprod_{p^*_{\Spc}i_*^{\Spc} \hspc[Z]{X_Z}} \hspc[Y]{Y}) \to p_\sharp^{\Spc}\hspc[Y]{Y}.
  \end{equation*}
If $k$ (resp. $q$) denotes the base change of $i$ (resp. $p$) along $p$ (resp. $i$), then by \remref{rem:h(X,t) smooth base change} and the smooth base change formula $p^*i_* \simeq k_*q^*$ (\propref{prop:smooth base change}), we see that this morphism is the image by $p_\sharp$ of the morphism
  \begin{equation} \label{eq:localization 6}
    \h_Y^{Y_Z}(X \fibprod_S Y) \fibprod_{k_*^{\Spc} \h_{Y_Z}((X \fibprod_S Y)_Z)} \hspc[Y]{Y} \to \hspc[Y]{Y}.
  \end{equation}
Now the source is nothing else than the space $\hspc[Y]{X \fibprod_S Y, t_Y}$, where $t_Y : Z \fibprod_S Y \to X \fibprod_S Y$ is the base change of $t$ along $p$.
Hence we conclude by \propref{prop:h(X,t) is contractible}.

\bibliographystyle{alphamod}

\newpage
\bibliography{references}

\end{document}